\documentclass[final]{siamltex}
\usepackage[left=1in,top=1in,right=1in,bottom=1in,letterpaper]{geometry}
\usepackage{amsmath,amssymb,amsfonts}
\usepackage[ruled,vlined]{algorithm2e}
\usepackage{cite}

\usepackage{subfigure}
\usepackage[usenames]{color}
\usepackage[final,notref,notcite]{showkeys}
\newif\ifpdf\ifx\pdfoutput\undefined\pdffalse\else\pdfoutput=1\pdftrue\fi
\usepackage[colorlinks=true, pdfstartview=FitV, linkcolor=blue,
                                    citecolor=blue, urlcolor=blue]{hyperref}

  \ifpdf        \usepackage[pdftex]{graphicx}
        \usepackage{epstopdf}
        \usepackage{url}
  \else\usepackage{graphicx}\fi

\newtheorem{remark}{Remark}[section]

\newcommand{\vc}{\mathrm{vec}}

\newcommand{\bm}[1]{\boldsymbol{#1}}

\newcommand{\eps}{\epsilon}

\newcommand{\st}{\text{ s.t. }}

\newcommand{\mbfa}{\mathbf{A}}
\newcommand{\mbfb}{\mathbf{B}}
\newcommand{\mbfc}{\mathbf{C}}

\newcommand{\mbfi}{\mathbf{I}}
\newcommand{\mbfm}{\mathbf{M}}

\newcommand{\mbfu}{\mathbf{U}}
\newcommand{\mbfx}{\mathbf{X}}

\newcommand{\mbr}{\mathbb{R}}

\newcommand{\mca}{\mathcal{A}}
\newcommand{\mcb}{\mathcal{B}}
\newcommand{\mcc}{\mathcal{C}}

\newcommand{\mcm}{\mathcal{M}}
\newcommand{\mcn}{\mathcal{N}}
\newcommand{\mcp}{\mathcal{P}}

\newcommand{\mcs}{\mathcal{S}}

\newcommand{\mcw}{\mathcal{W}}
\newcommand{\mcx}{\mathcal{X}}
\newcommand{\mcy}{\mathcal{Y}}
\newcommand{\mcz}{\mathcal{Z}}

\newcount\refnum\refnum=0
\def\myref{{\global\advance\refnum by 1} {\bf \large Lecture \the \refnum. }}
\newcommand{\bfx}{{\bf x}}
\newcommand{\bfa}{{\bf a}}
\newcommand{\bfb}{{\bf b}}

\newcommand{\sign}{{\mathrm{sign}}}

\DeclareMathOperator*{\argmin}{argmin}

\usepackage[normalem]{ulem}

\begin{document}

\title{Alternating proximal gradient method for \\sparse nonnegative Tucker decomposition}
\author{Yangyang Xu\thanks{\url{yangyang.xu@rice.edu}. Department of Applied and Computational Mathematics, Rice University, Houston, Texas.}
}

\date{\today}

\maketitle

\begin{abstract}Multi-way data arises in many applications such as electroencephal-ography (EEG) classification, face recognition, text mining and hyperspectral data analysis. Tensor decomposition has been commonly used to find the hidden factors and elicit the intrinsic structures of the multi-way data. This paper considers sparse nonnegative Tucker decomposition (NTD), which is to decompose a given tensor into the product of a core tensor and several factor matrices with sparsity and nonnegativity constraints. An alternating proximal gradient method (APG) is applied to solve the problem. The algorithm is then modified to sparse NTD with missing values. Per-iteration cost of the algorithm is estimated scalable about the data size, and global convergence is established under fairly loose conditions. Numerical experiments on both synthetic and real world data demonstrate its superiority over a few state-of-the-art methods for (sparse) NTD from partial and/or full observations. The MATLAB code along with demos are accessible from the author's homepage.
\end{abstract}

\section{Introduction}
A \emph{tensor} is a multi-dimensional array. For example, a \emph{vector} is a first-order {tensor}, and a \emph{matrix} is a second-order tensor. The order of a tensor is the number of dimensions, also called \emph{way} or \emph{mode}. Tensors naturally arise in the applications that collect data along multiple dimensions, including space, time, and spectrum, from different subjects (e.g., patients), under varying conditions, and in different modalities. They can also be created by tensorization of lower dimensional data \cite{cichockitensor2014}. Examples include medical data (CT, MRI, EEG), text data and hyperspectral images. An efficient approach to elicit the intrinsic structure of multi-dimensional data is tensor decomposition. Two commonly used tensor decompositions are CANDECOMP/PARAFAC decomposition (CPD) \cite{carroll1970analysis, harshman1970foundations} and Tucker decomposition (TD) \cite{tucker1966some}. CPD decomposes an $N$th-order tensor $\bm{\mcm}$ into the product of $N$ factor matrices $\mbfa_1,\cdots,\mbfa_N$, and TD decomposes $\bm{\mcm}$ into the product of a core tensor $\bm{\mcc}$ and $N$ factor matrices $\mbfa_1,\cdots,\mbfa_N$. 

This paper focuses on sparse nonnegative Tucker decomposition (NTD) \cite{kim2007nonnegative}, which imposes nonnegativity and uses $\ell_1$-regularization terms to promote sparsity structure on the core tensor and/or factor matrices. Nonnegativity allows only additivity, so the solutions are often intuitive to understand and explain. Promoting the sparsity of the core tensor aims at improving the interpretability of the  solutions. Roughly speaking, the core tensor interacts with all the factor matrices, and a simple one is often preferred \cite{kiers1998joint}. Consider a three-way tensor, for example. The $(1,1,1)$-th component of the core tensor couples the first columns of three factor matrices together. If it is not \emph{zero}, then the three columns interacts with each other. Otherwise, they have no or only weak relations. Forcing the core tensor to be sparse can often keep strong interactions between the factor matrices and remove the weak ones. Sparse factor matrices make the decomposed parts more meaningful and can enhance uniqueness as explained in \cite{morup2008algorithms}. Sparse NTD has found a large number of applications such as in EEG classification \cite{cong2012feature}, hyperspectral data analysis \cite{zhang2008tensor}, text mining \cite{morup2008algorithms}, face recoginition \cite{zafeiriou2009discriminant}, and so on.

\subsection{Related work}
NTD is a highly non-convex problem, and sparse regularizers make the problem even harder. A natural and often efficient way to solve the problem is to alternatingly update the core tensor and factor matrices. It includes, but not limited to, alternating least squares method (ALS) \cite{friedlander2008computing}, column-wise coordinate descent (CCD) \cite{liu2011sparse}, higher-order multiplicative update (HONMF) \cite{morup2008algorithms}, and hierarchical alternating least squares (HALS) \cite{phan2011extended}. ALS alternatingly updates the core tensor and factor matrices by solving a sequence of nonnegative least squares (NLS) problems, which requires to calculate matrix inverse and make ALS unsuitable for large-scale problems\footnote{There appears no exact definition of ``large-scale''. The concept can involve with the development of the computing power. Here, we roughly mean there are over millions of variables or data values.}. For this reason, \cite{friedlander2008computing} simply restricts the core tensor to be super-diagonal in its numerical tests. CCD has closed form update for each column of a factor matrix. However, to update the core tensor, it still requires to solve a big NLS problem, which makes CCD unsuitable for large-scale problems either. HONMF is an extension of the multiplicative update method in \cite{lee2001algorithms} for nonnegative matrix factorization \cite{paatero1994positive, lee1999learning} and has a relatively low per-iteration cost. At each iteration, it only needs some tensor-matrix multiplications and component-wise divisions. The drawback of HONMF is its slow convergence, which makes the algorithm often run a large number of iterations to reach an acceptable data fitting. Like ALS, HALS needs to solve a sequence of NLS problems, but it updates factor matrices in a column-wise way and the core tensor component-wisely, which enables closed form solutions for all subproblems. In addition, HALS often converges faster than HONMF. However, as shown in \cite{phan2011damped}, the convergence speed of HALS is still not satisfying.

There are also algorithms that update the core tensor and factor matrices simultaneously, such as the damped Gauss-Newton method (dGN) in \cite{phan2011damped}.  It is demonstrated that dGN overwhelmingly outperforms HONMF and HALS in terms of convergence speed. 

Recently, \cite{xublock} proposed an alternating proximal gradient method (APG) for solving NCP, and it was observed superior to some other algorithms such as the alternating direction method of multiplier (ADMM) \cite{zhang2010admnmf} and alternating nonnegative least squares method (ANLS) \cite{kim2008non, kim2008toward}  in both speed and solution quality. Unlike ANLS that exactly solves each subproblem, APG updates every factor matrix by solving a relaxed subproblem with a separable quadratic objective. Each relaxed subproblem has a closed form solution, which makes low per-iteration cost. Using an extrapolation technique, APG also converges very fast. 

\subsection{Overview of tensor}\label{sec:tensor}

\textbf{Notation.} We use small letters $a,x,\cdots$ for scalars, bold small letters $\bfa, \bfx,\cdots$ for vectors, bold capital letters $\mbfa,\mbfb,\cdots$ for matrices and bold caligraphic letters $\bm{\mcc},\bm{\mcm},\cdots$ for tensors. The components of a tensor $\bm{\mcx}$ are written in the form of $x_{i_1i_2\cdots i_N}$, which denotes the $(i_1,i_2,\cdots,i_N)$-th component of $\bm{\mcx}$.

Before proceeding with the model, we overview some tensor related concepts. For more details,  we refer the readers to the nice review paper \cite{kolda2009tensor}.
\begin{itemize}
\item A \emph{fiber} of $\bm{\mcx}$ is a vector obtained by fixing all indices of $\bm{\mcx}$ except one. 
\item The \emph{vectorization} of $\bm{\mcx}$ gives a vector, which is obtained by stacking all mode-1 fibers of $\bm{\mcx}$ and denoted by $\vc(\bm{\mcx})$. 
\item The mode-$n$ \emph{matricization} of $\bm{\mcx}$ is a matrix denoted by $\mbfx_{(n)}$ whose columns are mode-$n$ fibers of $\bm{\mcx}$ in the lexicographical order. 
\item The mode-$n$ product of $\bm{\mcx}\in\mbr^{I_1\times\cdots\times I_N}$ with $\mbfa\in\mbr^{J\times I_n}$ is written as $\bm{\mcx}\times_n\mbfa\in\mbr^{I_1\times \cdots \times I_{n-1}\times J\times I_{n+1}\times \cdots\times I_N}$, defined component-wisely by
\begin{equation*}\label{eq:tm}
(\bm{\mcx}\times_n \mbfa)_{i_1\cdots i_{n-1}ji_{n+1}\cdots i_N}=\sum_{i_n=1}^{I_n}x_{i_1i_2\cdots i_N}a_{ji_n}.
\end{equation*} 
\item The \emph{inner product} of $\bm{\mca},\bm{\mcb}\in\mbr^{I_1\times\cdots\times I_N}$ is
$\langle\bm{\mca},\bm{\mcb}\rangle\triangleq\sum_{i_1,\cdots,i_N}a_{i_1\cdots i_N}b_{i_1\cdots i_N}.$
The Frobenious norm of $\bm{\mcx}$ is 
$\|\bm{\mcx}\|_F\triangleq\sqrt{\langle\bm{\mcx},\bm{\mcx}\rangle}.$
\item Given $\bm{\mcm}\in\mbr^{I_1\times\cdots\times I_N}$, the \emph{Tucker decomposition} of $\bm{\mcm}$ is to find a core tensor $\bm{\mcc}\in\mbr^{R_1\times\cdots\times R_N}$ with $R_n\le I_n,\forall n$ and $N$ factor matrices $\mbfa_n\in\mbr^{I_n\times R_n},n=1,\cdots,N$ such that \begin{equation}\label{eq:tucker}
\bm{\mcm}\approx\bm{\mcc}\times_1\mbfa_1\cdots\times_N\mbfa_N.
\end{equation} 
\end{itemize}
 
It is not difficult to verify that
if $\bm{\mcx}=\bm{\mcc}\times_1\mbfa_1\cdots\times_N\mbfa_N$, then 
\begin{equation}\label{eq:vec}
\vc(\bm{\mcx})=\big(\otimes_{n=N}^1\mbfa_n\big)\vc(\bm{\mcc}),
\end{equation}
where 
\begin{equation}\label{eq:kron}
\otimes_{n=N}^1\mbfa_n\triangleq\mbfa_N\otimes\cdots\otimes\mbfa_1,
\end{equation}
 and $\mbfa\otimes\mbfb$ denotes the Kronecker product of $\mbfa$ and $\mbfb$. In addition,
\begin{equation}\label{eq:mat}\mbfx_{(n)}=\mbfa_n\mbfc_{(n)}\left(\otimes_{\substack{i=N\\i\neq n}}^1\mbfa_i\right)^\top.
\end{equation}

\subsection{Contributions}
We apply and improve the APG method proposed in \cite{xublock} to the sparse NTD problem
\begin{equation}\label{eq:spntd}
\begin{array}{l}
\underset{\bm{\mcc},\mbfa}{\min}\ F(\bm{\mcc},
\mbfa)\equiv\ell(\bm{\mcc},
\mbfa)+\lambda_c\|\bm{\mcc}\|_1+\overset{N}{\underset{n=1}{\sum}}\lambda_n\|\mbfa_n\|_1,\\[0.2cm]
\st \bm{\mcc}\in\mbr^{R_1\times\cdots\times R_N}_+, \mbfa_n\in\mbr^{I_n\times R_n}_+, ~n = 1,\cdots,N,
\end{array}
\end{equation}
where $\mbr^{I_n\times R_n}_+$ contains all $I_n\times R_n$ matrices with nonnegative components, $\mbfa$ denotes $(\mbfa_1,\cdots,\mbfa_N)$, $$\ell(\bm{\mcc},
\mbfa)=\frac{1}{2}\|\bm{\mcc}\times_1\mbfa_1\cdots\times_N\mbfa_N-\bm{\mcm}\|_F^2$$
is a data fitting term that measures the approximation in \eqref{eq:tucker}, $\bm{\mcm}\in\mbr^{I_1\times\cdots\times I_N}_+$ is a given tensor, $\|\bm{\mcc}\|_1\triangleq\sum_{i_1,\cdots,i_N}|c^{}_{i_1\cdots i_N}|$ is used to promote the sparsity of $\bm{\mcc}$, and $\lambda_c,\lambda_1,\cdots,\lambda_N$ are parameters balancing the data fitting and sparsity level. 

Our algorithm iteratively updates the core tensor $\bm{\mcc}$ and factor matrices alternatingly in the order of $\bm{\mcc},\mbfa_1,\bm{\mcc},\mbfa_2,\cdots,\bm{\mcc},\mbfa_N$. We analyze the algorithm's per-iteration complexity and give its global convergence. The algorithm is modified to sparse NTD with missing values. We also consider some extensions of NTD including sparse higher-order principal component analysis \cite{allen2012sparse}. Our algorithm is carefully implemented in MATLAB and compared to some state-of-the-art methods for solving (sparse) NTD from partial and/or full observations on both synthetic and real world data. Numerical results show that the proposed algorithm makes superior performance over all the compared ones in almost all cases.

\subsection{Outline} The rest of the paper is organized as follows. Section \ref{sec:spntd} applies APG to sparse NTD problem. The algorithm is modified for sparse NTD with missing values in section \ref{sec:spntdc}, and some extensions are considered in section \ref{sec:hopca}. Numerical results are shown in section \ref{sec:numerical}. Finally, section \ref{sec:discussion} concludes the paper.

\section{Sparse nonnegative Tucker decomposition}\label{sec:spntd}

\subsection{Bound constraints for well-definedness} 
Note that for any positive scalars $s_c,s_1,\cdots,s_N$ such that their product equals \emph{one}, $(s_c\bm{\mcc}, s_1\mbfa_1, \cdots, s_N\mbfa_N)$ does not change the value of $\ell$. Hence, if some $\lambda$'s vanish, the corresponding variables would be unbounded such that the variables with positive $\lambda$'s would approach to \emph{zero}, and \eqref{eq:spntd} may not admit a solution. To tackle this problem, if $\lambda_n=0$, we add 
\begin{equation}\label{eq:con-a}
\mbfa_n\le \max(1,\|\bm{\mcm}\|_\infty)
\end{equation} to bound $\mbfa_n$, where $\|\bm{\mcm}\|_\infty$ denotes the maximum component of $\bm{\mcm}$. If $\lambda_c=0$, we add
\begin{equation}\label{eq:con-c}
\bm{\mcc}\le \max(1,\|\bm{\mcm}\|_\infty)
\end{equation}
to bound $\bm{\mcc}$. The constraints in \eqref{eq:con-a} and \eqref{eq:con-c} are reasonable according to the following proposition, which is not difficult to show.

\begin{proposition}\label{prop:con}
If $\bm{\mcm}=\tilde{\bm{\mcc}}\times_1\tilde{\mbfa}_1\cdots\times_N\tilde{\mbfa}_N$ for some $(\tilde{\bm{\mcc}},\tilde{\mbfa}_1,\cdots,\tilde{\mbfa}_N)$, then there exists some $({\bm{\mcc}},{\mbfa}_1,\cdots,{\mbfa}_N)$ satisfying \eqref{eq:con-a} and \eqref{eq:con-c} such that $\bm{\mcm}={\bm{\mcc}}\times{\mbfa}_1\cdots\times_N{\mbfa}_N$ and $({\bm{\mcc}},{\mbfa}_1,\cdots,{\mbfa}_N)$ has the same sparsity as that of  $(\tilde{\bm{\mcc}},\tilde{\mbfa}_1,\cdots,\tilde{\mbfa}_N)$.
\end{proposition}

\begin{remark}\label{rm:bound-con}
If $\tilde{\bm{\mcc}}\times_1\tilde{\mbfa}_1\cdots\times_N\tilde{\mbfa}_N$ is not exactly equal but close to $\bm{\mcm}$, one can magnify the bounds in \eqref{eq:con-a} and \eqref{eq:con-c} by multiplying some $\tau>1$.
\end{remark}

\subsection{APG for sparse NTD}
For convenience, we assume all $\lambda$'s to be positive in the derivation of our algorithm, so there are no constraints as in \eqref{eq:con-a} and \eqref{eq:con-c} present. Our algorithm is based on the APG method proposed in \cite{xublock}.
Suppose the current iterate is $(\tilde{\bm{\mcc}},\tilde{\mbfa})$.
We update $\bm{\mcc}$ by
\begin{align}
\bm{\mcc}_{\mathrm{new}}&=\argmin_{\bm{\mcc}\ge0}\langle\nabla_{\bm{\mcc}}\ell(\hat{\bm{\mcc}},
\tilde{\mbfa}),\bm{\mcc}-\hat{\bm{\mcc}}\rangle+\frac{L_c}{2}\|\bm{\mcc}-\hat{\bm{\mcc}}\|_F^2+\lambda_c\|\bm{\mcc}\|_1,\label{core-prob}\\
&= \max\left(0,\hat{\bm{\mcc}}-\frac{1}{L_c}\nabla_{\bm{\mcc}}\ell(\hat{\bm{\mcc}},
\tilde{\mbfa})-\frac{\lambda_c}{L_c}\right),\label{eq:core}
\end{align}
where $L_c$ is a Lipschitz constant of $\nabla_{\bm{\mcc}}\ell(\bm{\mcc},
\tilde{\mbfa})$ with respect to $\bm{\mcc}$, namely,
$$\|\nabla_{\bm{\mcc}}\ell(\bm{\mcc}_1,
\tilde{\mbfa})-\nabla_{\bm{\mcc}}\ell(\bm{\mcc}_2,
\tilde{\mbfa})\|_F\le L_c\|\bm{\mcc}_1-\bm{\mcc}_2\|_F,\ \forall ~\bm{\mcc}_1, \bm{\mcc}_2,$$
and $\hat{\bm{\mcc}}$ is an extrapolated point.
Similarly, if the current iterate is $(\tilde{\bm{\mcc}},\tilde{\mbfa})$, a factor matrix $\mbfa_n$ is updated by
\begin{align}
(\mbfa_n)_{\mathrm{new}}=&\argmin_{\mbfa_n\ge0}\langle\nabla_{\mbfa_n}\ell(\tilde{\bm{\mcc}},
\tilde{\mbfa}_{j<n},\hat{\mbfa}_n,\tilde{\mbfa}_{j>n}),\mbfa_n-\hat{\mbfa}_n\rangle\label{factor-prob}\\
& \hspace{0.3cm}+\frac{L_n}{2}\|\mbfa_n-\hat{\mbfa}_n\|_F^2+\lambda_n\|\mbfa_n\|_1,\cr
=& \max\left(0,\hat{\mbfa}_n-\frac{1}{L_n}\nabla_{\mbfa_n}\ell(\tilde{\bm{\mcc}},
\tilde{\mbfa}_{j<n},\hat{\mbfa}_n,\tilde{\mbfa}_{j>n})-\frac{\lambda_n}{L_n}\right),\label{eq:factor}
\end{align}
where $L_n$ is a Lipschitz constant of $\nabla_{\mbfa_n}\ell(\tilde{\bm{\mcc}},
\tilde{\mbfa}_{j<n},{\mbfa}_n,\tilde{\mbfa}_{j>n})$ with respect to $\mbfa_n$, and $\hat{\mbfa}_n$ is an extrapolated point.

One can perform \eqref{eq:core} and \eqref{eq:factor} to update $\bm{\mcc}$ and $\mbfa$ in different manners. Directly applying the APG method proposed in \cite{xublock} leads to the order of $\bm{\mcc},\mbfa_1,\cdots,\mbfa_N$. However, since the core tensor $\bm{\mcc}$ interacts with all $\mbfa_n$'s, updating it more frequently is expected to speed up the convergence of the algorithm. Hence, a more efficient way would be to update the variables in the order of $\bm{\mcc},\mbfa_1,\bm{\mcc},\mbfa_2,\cdots,\bm{\mcc},\mbfa_N$. Figure \ref{fig:comp_2orders} shows the convergence behavior of APG with two different updating orders on a synthetic tensor and the Swimmer dataset \cite{donoho2003}. From the figure, we see that APG with the updating order $\bm{\mcc},\mbfa_1,\cdots,\mbfa_N$ performs comparably well as that with the order $\bm{\mcc},\mbfa_1,\bm{\mcc},\mbfa_2,\cdots,\bm{\mcc},\mbfa_N$ on the randomly generated data. However, the former behaves much worse than the latter on the Swimmer dataset. For this reason, we only consider the latter one, whose pseudocode is shown in Algorithm \ref{alg:apg}.

\begin{figure}\caption{Results by APG with two different orders of updating the core tensor and factor matrices. (a). APG on a Gaussian random $20\times20\times20\times20$ tensor $\bm{\mcm}$ with core size $5\times5\times5\times5$; (b). APG on the $32\times32\times256$ Swimmer dataset \cite{donoho2003} with core size $24\times20\times20$.}
\label{fig:comp_2orders}
\begin{center}
\subfigure[random dataset]{
\begin{minipage}[t]{0.35\textwidth}
\includegraphics[width=0.99\textwidth]{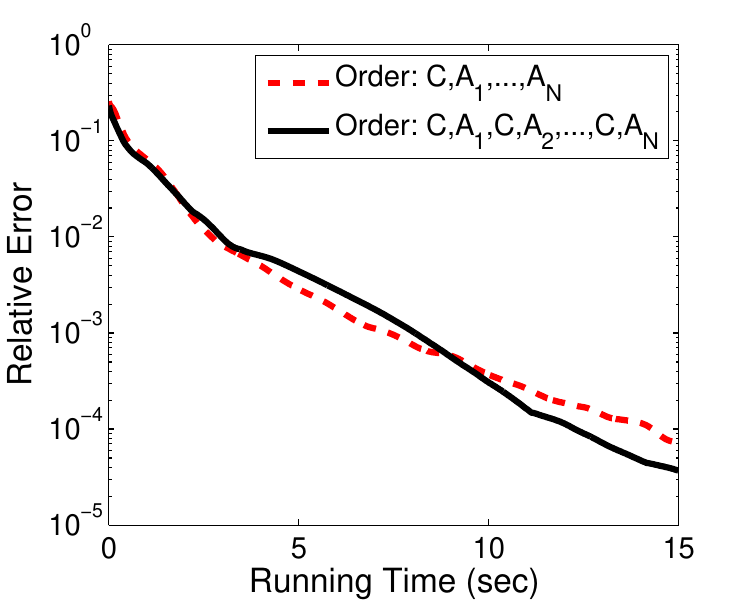}
\end{minipage}
}
\subfigure[Swimmer dataset]{
\begin{minipage}[t]{0.35\textwidth}
\includegraphics[width=0.99\textwidth]{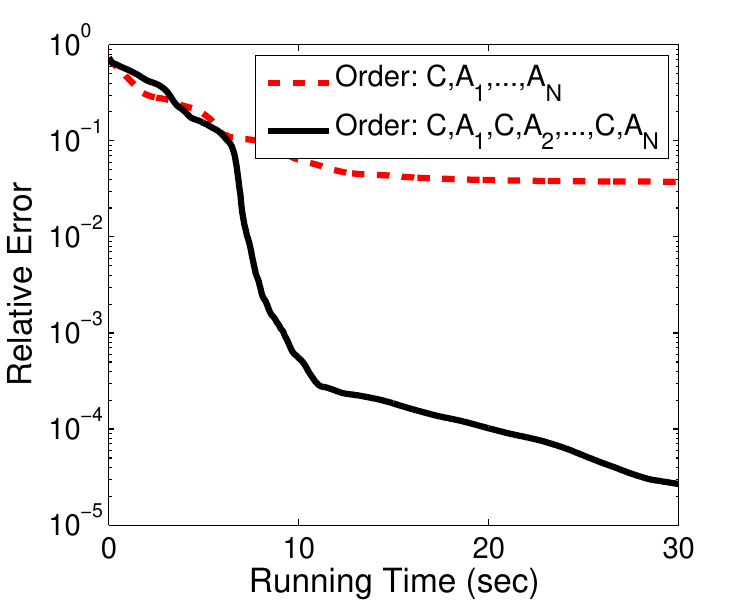}
\end{minipage}
}
\end{center}
\end{figure}

\begin{algorithm}\caption{Alternating proximal gradient for sparse NTD}\label{alg:apg}
\DontPrintSemicolon
{\small 
\KwData{tensor $\bm{\mcm}$, core dimension $(R_1,\cdots,R_N)$, parameters $\lambda_c,\lambda_1,\cdots,\lambda_N\ge 0$, and $(\bm{\mcc}^{-1},\mbfa^{-1})=(\bm{\mcc}^{0},\mbfa^{0})$.}
\For{$k=1,2,\cdots$}{
Set $\bm{\mcc}^{k,-1}=\bm{\mcc}^{k,0}=\bm{\mcc}^0$ if $k=1$ and $\bm{\mcc}^{k,-1}=\bm{\mcc}^{k-1,N-1},\ \bm{\mcc}^{k,0}=\bm{\mcc}^{k-1,N}$ otherwise.\;
\For{$n=1,\cdots,N$}{
Choose $L_c^{k,n}$ to be a Lipschitz constant of $\nabla_{\bm{\mcc}}\ell(\bm{\mcc},\mbfa_{j<n}^k,\mbfa_{j\ge n}^{k-1})$ about $\bm{\mcc}$.\;
Choose $\omega_c^{k,n}\ge0$ and set $\hat{\bm{\mcc}}^{k,n}={\bm{\mcc}}^{k,n-1}+\omega_c^{k,n}({\bm{\mcc}}^{k,n-1}-{\bm{\mcc}}^{k,n-2})$.\;
Update $\bm{\mcc}$ by
\begin{equation}\label{update-c}
{\bm{\mcc}}^{k,n}= \max\left(0,\hat{\bm{\mcc}}^{k,n}-\frac{1}{L_c^{k,n}}\nabla_{\bm{\mcc}}\ell(\hat{\bm{\mcc}}^{k,n},
\mbfa_{j<n}^k,\mbfa_{j\ge n}^{k-1})-\frac{\lambda_c}{L_c^{k,n}}\right);
\end{equation}
Choose $L_n^k$ to be a Lipschitz constant of $\nabla_{\mbfa_n}\ell(\bm{\mcc}^{k,n},\mbfa_{j<n}^k,\mbfa_n,\mbfa_{j< n}^{k-1})$ about $\mbfa_n$.\;
Choose $\omega_n^{k}\ge0$ and set $\hat{\mbfa}_n^k=\mbfa_n^{k-1}+\omega_n^k(\mbfa_n^{k-1}-\mbfa_n^{k-2})$.\;
Update $\mbfa_n$ by 
\begin{equation}\label{update-a}
\mbfa_n^k=\max\left(0,\hat{\mbfa}_n^k-\frac{1}{L_n^k}\nabla_{\mbfa_n}\ell(\bm{\mcc}^{k,n},
\mbfa_{j<n}^k,\hat{\mbfa}_n^k,\mbfa_{j>n}^{k-1})-\frac{\lambda_n}{L_n^k}\right).
\end{equation}
\If{$F({\bm{\mcc}}^{k,n},\mbfa_{j\le n}^k, \mbfa_{j>n}^{k-1})>F({\bm{\mcc}}^{k,n-1},\mbfa_{j<n}^k, \mbfa_{j\ge n}^{k-1})$}
{
\lnlset{reupdate}{ReDo}Re-update ${\bm{\mcc}}^{k,n}$ and $\mbfa_n^k$ by \eqref{update-c} and \eqref{update-a} with $\hat{\bm{\mcc}}^{k,n}={\bm{\mcc}}^{k,n-1}$ and $\hat{\mbfa}_n^k=\mbfa_n^{k-1}$, respectively.\;
}
} 
Set $\bm{\mcc}^k=\bm{\mcc}^{k,N}$.\;
\If{Some stopping conditions are satisfied}{
Output $(\bm{\mcc}^k,\mbfa_1^k,\cdots,\mbfa_N^k)$ and stop.
}
}
}
\end{algorithm}

\begin{remark}\label{rm:alg}
We do re-update in Line \textbf{\ref{reupdate}} to make the objective nonincreasing. The monotonicity of the objective is important since the algorithm may perform unstably without the re-update. The computational cost of one objective evaluation is much cheaper than, actually not in the same order as, one gradient
computation. Detailed complexity analysis is listed in Appendix \ref{sec:complexity}. Moreover, in each one of our experiments, the re-update occurs only a few times (often less than 10), so it needs only a little more computations.

If some $\lambda_n$ and/or $\lambda_c$ vanish, we further do projections 
\begin{equation}\label{proj-c}
\bm{\mcc}^{k,n}=\min\left(\max(1,\|\bm{\mcm}\|_\infty),\bm{\mcc}^{k,n}\right)
\end{equation} 
after \eqref{update-c} and 
\begin{equation}\label{proj-a}
\mbfa_n^k=\min\left(\max(1,\|\bm{\mcm}\|_\infty),\mbfa_n^k\right)
\end{equation} 
after \eqref{update-a}. Omitting the superscript, it is easy to show that \eqref{proj-c} and \eqref{proj-a} respectively solve \eqref{core-prob} and \eqref{factor-prob} with the extra constraints \eqref{eq:con-c} and \eqref{eq:con-a}.
\end{remark}

\subsection{Parameter settings}\label{sec:param-set} 
In our implementation of Algorithm \ref{alg:apg}, we set 
{\small$$L_c^{k,n}=\max\big(1, \left\|(\mbfa_N^{k-1})^\top \mbfa_N^{k-1}\otimes\cdots\otimes (\mbfa_n^{k-1})^\top \mbfa_n^{k-1}\otimes (\mbfa_{n-1}^{k})^\top \mbfa_{n-1}^{k}\otimes\cdots\otimes(\mbfa_{1}^{k})^\top \mbfa_{1}^{k}\right\|\big),$$}where $\|\cdot\|$ denotes matrix operator norm. Note that computing $L_c^{k,n}$ does not need to form the expensive Kronecker product because 
\begin{equation*}
\left\|\mbfa_N^\top \mbfa_N\otimes\cdots\otimes\mbfa_{1}^\top \mbfa_{1}\right\|
=\prod_{i=1}^N\left\|\mbfa_{i}^\top \mbfa_{i}\right\|.
\end{equation*}
In the same way, we set  
\begin{equation}\label{eq:La}
L_n^k=\max\big(1, \|\mbfb_n^k(\mbfb_n^k)^\top\|\big),
\end{equation}
where
 \begin{equation}\label{eq-b}
\mbfb_n^k=\mbfc_{(n)}^{k,n}\left(\mbfa_N^{k-1}\otimes\cdots\otimes\mbfa_{n+1}^{k-1}
\otimes\mbfa_{n-1}^k\otimes\cdots\otimes\mbfa_1^k\right)^\top.\end{equation}

In addition, we take
\begin{equation}\label{eq:weightc}
\omega_c^{k,n}=\min\left(\hat{\omega}_c^{k,n},0.9999\sqrt{\frac{L_c^{k,n-1}}{L_c^{k,n}}}\right),
\end{equation}
where $\hat{\omega}_c^{k,n}$ follows
\begin{subequations}\label{eq:fista-c}
\begin{align}
&\hat{\omega}_c^{k,n}=\frac{t^{k,n-1}-1}{t^{k,n}},\label{eq:fista-c1}\\
&t^{1,0}_c=1,\  t^{k,0}_c = t^{k-1,N}_c, \text{ for }k\ge 2, \\
& t^{k,n}_c=\frac{1}{2}\left(1+\sqrt{1+4(t^{k,n-1}_c)^2}\right), \text{ for }k\ge 1, n =1,\cdots,N.
\end{align}
\end{subequations}
In the same way,
\begin{equation}\label{eq:weighta}
\omega_n^{k}=\min\left(\hat{\omega}^{k},0.9999\sqrt{\frac{L_n^{k-1}}{L_n^{k}}}\right),
\end{equation}
where $\hat{\omega}^{k}$ follows
\begin{subequations}\label{eq:fista-a}
\begin{align}
&\hat{\omega}^{k}=\frac{t^{k-1}-1}{t^{k}},\label{eq:fista-a1}\\
&t^{0}=1,\  t^{k}=\frac{1}{2}\left(1+\sqrt{1+4(t^{k-1})^2}\right), \text{ for }k\ge 1.
\end{align}
\end{subequations}

\begin{remark}We perform ``min'' operation in \eqref{eq:weightc} and \eqref{eq:weighta} for convergence; see Theorem \ref{thm:convg}. The weights $\hat{\omega}_c^{k,n}$ in \eqref{eq:fista-c} and $\hat{\omega}^k$ in \eqref{eq:fista-a} are the same as that used in \cite{BeckTeboulle2009} for convex problems. Numerically, we observe that the extrapolation technique using the weights given in \eqref{eq:weightc} and \eqref{eq:weighta} can significantly speed up our algorithm. We also tested APG with the dynamically updated weight used in \cite{wen2012solving, Ling-Xu-Yin-mtx-conf-11} for non-convex matrix completion problem and observed that APG performs as well as that with the above extrapolation weights.
\end{remark}

\subsection{Per-iteration complexity} 
Suppose $\bm{\mcm}\in\mathbb{R}^{I_1\times\ldots\times I_N}$ and the core tensor $\bm{\mcc}\in\mathbb{R}^{R_1\times\ldots\times R_N}$. Then the per-iteration cost of Algorithm \ref{alg:apg} is roughly
\begin{equation}\label{eq:iter-cost}
N\cdot\mathcal{O}\left(\sum_{j=1}^N\big(\prod_{i=1}^j R_i\big)\big(\prod_{i=j}^N I_i\big)+\sum_{j=1}^N\big(\prod_{i=1}^j I_i\big)\big(\prod_{i=j}^N R_i\big)\right).
\end{equation}
The detailed analysis is given in Appendix \ref{sec:complexity}.

\begin{remark}
 If $N=\mathcal{O}(1)$ and $\max_nR_n\le \mathcal{O}(\log\prod_{i=1}^NI_i)$, then the per-iteration cost of Algorithm \ref{alg:apg} is scalable\footnote{Here, by scalability, we mean the cost is no greater than $s\cdot\log(s)$ if the data size is $s$.} about the data size $\prod_{i=1}^NI_i$. 
\end{remark}

\subsection{Convergence results}\label{sec:convg}

It is shown in \cite{xublock} that the APG method with cyclic block updating rule has global convergence to a stationary point. Since Algorithm \ref{alg:apg} uses a different block updating order, its convergence cannot be directly obtained from \cite{xublock}. However, we can still obtain the global convergence\footnote{Since the problem is non-convex, we only get convergence to a stationary point, and different starting points can produce different limit points.}, which is summarized in Theorem \ref{thm:convg}. Although the proof idea for Theorem \ref{thm:convg} is similar to that in \cite{xublock}, some places need careful modifications. Hence, for completeness, we include a modified proof in the Appendix. 

\begin{theorem}\label{thm:convg}
Let $\big\{\bm{\mcw}^k\triangleq(\bm{\mcc}^k,\mbfa^k)\big\}$ be the sequence generated by Algorithm \ref{alg:apg}. If $\lambda_c,\lambda_1,\cdots,\lambda_N$ are all positive, and
\begin{enumerate}
\item There exist positive constants $L_d, L_u$ such that $L_c^{k,n}, L_n^k \in [L_d, L_u]$;  
\item There is a positive constant $\delta_\omega<1$ such that $\omega_c^{k,n}\le\delta_\omega\sqrt{\frac{L_c^{k,n-1}}{L_c^{k,n}}}$ and $\omega_n^k\le\delta_\omega\sqrt{\frac{L_n^{k-1}}{L_n^k}}$ for all $n$ and $k$, where we use the notation $L_c^{k,0}=L_c^{k-1,N}$;
\end{enumerate}
 then $\bm{\mcw}^k$ converges to a stationary point $\bar{\bm{\mcw}}$ of \eqref{eq:spntd}. 
\end{theorem}

\begin{remark}\label{rm:convg}
Positivity of sparse parameters implies the boundedness of $\{\bm{\mcw}^k\}$, and thus the existence of $L_d$ and $L_u$ can be guaranteed if $L_c^{k,n}$ and $L_n^k$ are taken as in section \ref{sec:param-set}.
\end{remark}

\section{Sparse nonnegative Tucker decomposition with missing values}\label{sec:spntdc}
For some applications, $\bm{\mcm}$ may not be fully observed. This section modifies Algorithm \ref{alg:apg} to handle this case. The problem is formulated as
{\small\begin{equation}\label{eq:spntdc}
\begin{array}{l}
\underset{\bm{\mcc},\mbfa}{\min}\ F_\Omega(\bm{\mcc},\mbfa)\equiv\frac{1}{2}\|\mcp_\Omega(\bm{\mcc}\times_1\mbfa_1\cdots\times_N\mbfa_N-\bm{\mcm})\|_F^2+\lambda_c\|\bm{\mcc}\|_1+\overset{N}{\underset{n=1}{\sum}}\lambda_n\|\mbfa_n\|_1,\\[0.2cm]
\st \bm{\mcc}\in\mbr^{R_1\times\cdots\times R_N}_+, \mbfa_n\in\mbr^{I_n\times R_n}_+, ~n = 1,\cdots,N,
\end{array}
\end{equation}
}
where $\Omega$ indexes the observed entries of $\bm{\mcm}$, and $\mcp_\Omega(\bm{\mca})$ keeps the entries of $\bm{\mca}$ in $\Omega$ and zeros out all others. As did in \cite{wen2012solving, Xu-Yin-Wen-Zhang-11}, we introduce variable $\bm{\mcx}$, restrict $\mcp_\Omega(\bm{\mcx})=\mcp_\Omega(\bm{\mcm})$, and write \eqref{eq:spntdc} equivalently to 
{\small\begin{equation}\label{eq:spntdc1}
\begin{array}{cl}
\underset{\bm{\mcc},\mbfa,\bm{\mcx}}{\min}& \frac{1}{2}\|\bm{\mcc}\times_1\mbfa_1\cdots\times_N\mbfa_N-\bm{\mcx}\|_F^2+\lambda_c\|\bm{\mcc}\|_1+\overset{N}{\underset{n=1}{\sum}}\lambda_n\|\mbfa_n\|_1,\\[0.2cm]
\st& \bm{\mcc}\in\mbr^{R_1\times\cdots\times R_N}_+, \mbfa_n\in\mbr^{I_n\times R_n}_+, ~n = 1,\cdots,N,~ \mcp_\Omega(\bm{\mcx})=\mcp_\Omega(\bm{\mcm}).
\end{array}
\end{equation}}

To modify Algorithm \ref{alg:apg} for \eqref{eq:spntdc} or equivalently \eqref{eq:spntdc1}, we set $\bm{\mcx}^0=\mcp_\Omega(\bm{\mcm})$ in the beginning. At the $k$-th iteration of Algorithm \ref{alg:apg}, we use $\bm{\mcm}=\bm{\mcx}^{k-1}$, wherever $\bm{\mcm}$ is referred to. 
After Line \textbf{\ref{reupdate}} of Algorithm \ref{alg:apg}, update $\bm{\mcx}$ by
\begin{equation}\label{update-x}
\bm{\mcx}^k=\mcp_\Omega(\bm{\mcm})+\mcp_{\Omega^c}(\bm{\mcc}^k\times_1\mbfa_1^k\cdots\times_N\mbfa_N^k).
\end{equation}

Compared to Algorithm \ref{alg:apg}, the modified method needs extra computation for the update \eqref{update-x}, which costs about 
$2\sum_{j=1}^N\big(\prod_{i=1}^jI_i\big)\big(\prod_{i=j}^NR_i\big).$
Therefore, the per-iteration complexity of the modified algorithm is still scalable about the data size if $N=\mathcal{O}(1)$ and $\max_nR_n\le \mathcal{O}(\log\prod_{i=1}^NI_i)$. In addition, following the proof of Theorem \ref{thm:convg}, one can show that the same convergence result holds for the modified algorithm. 

\section{Extensions}\label{sec:hopca}
For some applications, the core tensor $\bm{\mcc}$ may not be required nonnegative \cite{ding2010convex}. Algorithm \ref{alg:apg} can be modified to handle this case by changing \eqref{update-c} to
\begin{equation}\label{shrink-c}{\bm{\mcc}}^{k,n}= \mcs_{\frac{\lambda_c}{L_c^{k,n}}}\left(\hat{\bm{\mcc}}^{k,n}-\frac{1}{L_c^{k,n}}\nabla_{\bm{\mcc}}\ell(\hat{\bm{\mcc}}^{k,n},
\mbfa_{j<n}^k,\mbfa_{j\ge n}^{k-1})\right),\end{equation}
where $\mcs_\mu(\bm{\mcx})$ is a soft-thresholding operator defined component-wisely as $$\mcs_\mu(x)=\sign(x)\cdot\max(0,|x|-\mu).$$

The APG method can also be adapted to solve sparse higher-order principal component analysis (HOPCA), which imposes orthogonality constraint on each factor matrix. The problem is formulated as
\begin{equation}\label{eq:hopca}
\begin{array}{l}
\underset{\bm{\mcc},\mbfa}{\min}~\frac{1}{2}\|\bm{\mcc}\times_1\mbfa_1\cdots\times_N\mbfa_N-\bm{\mcm}\|_F^2+\lambda_c\|\bm{\mcc}\|_1+\overset{N}{\underset{n=1}{\sum}}\lambda_n\|\mbfa_n\|_1,\\
\st \mbfa_n^\top\mbfa_n=\mbfi_n,~n=1,\cdots,N,
\end{array}
\end{equation}
where $\mbfi_n$ is an identity matrix of appropriate size. When $\lambda_c=0$, the optimal $\bm{\mcc}=\bm{\mcm}\times_1\mbfa_1^\top\cdots\times_N\mbfa_N$, and one can eliminate $\bm{\mcc}$ as shown in \cite{kolda2009tensor}. The concurrency of sparsity and orthogonality constraints makes the problem much more difficult. The work \cite{allen2012sparse} considers rank-1 factor matrix with only one column and relaxes the orthogonality constraint to $\mbfa_n^\top\mbfa_n\le 1$. Then it applies block coordinate minimization method to solve the relaxed problem. When some $\mbfa_n$ has more than one columns, we relax \eqref{eq:hopca} to
{\small
\begin{equation}\label{eq:rhopca}
\hspace{-0.1cm}\begin{array}{l}
\underset{\bm{\mcc},\mbfa}{\min}~\frac{1}{2}\|\bm{\mcc}\times_1\mbfa_1\cdots\times_N\mbfa_N-\bm{\mcm}\|_F^2+\lambda_c\|\bm{\mcc}\|_1+\overset{N}{\underset{n=1}{\sum}}\lambda_n\|\mbfa_n\|_1+\frac{\mu}{2}\overset{N}{\underset{n=1}{\sum}}\underset{i\neq j}{\sum}\left(\bfa_{n,i}^\top\bfa_{n,j}\right)^2\\
\st \|\bfa_{n,j}\|_2\le 1,~n=1,\cdots,N, \forall j,
\end{array}\hspace{-0.3cm}
\end{equation}
}where $\bfa_{n,j}$ denotes the $j$-th column of $\mbfa_n$, $\sum_{i\neq j}\left(\bfa_{n,i}^\top\bfa_{n,j}\right)^2$ is used to promote the orthogonality of $\mbfa_n$, and $\mu$ is a penalty parameter. We want to mention that our orthogonality regularization term is similar to that used in \cite{ramirez2010classification} for promoting the discrepancy of dictionaries and also that used on pp. 222 of \cite{cichocki2009nonnegative}.

Our method for  \eqref{eq:rhopca} is similar to Algorithm \ref{alg:apg} and cycles over the variables by $\bm{\mcc},\mbfa_1,\bm{\mcc},\mbfa_2,\cdots,\bm{\mcc},\mbfa_N$. The update of $\bm{\mcc}$ is done by \eqref{shrink-c}, and $\mbfa_n$ is updated one column by one column. Specifically, assume the current iterate is $(\bm{\mcc}^{k,n},\mbfa_{i<n}^k,\mbfa_{i\ge n}^{k-1})$. Let $\mbfb_n^k$ be the one obtained from \eqref{eq-b}. Using \eqref{eq-ell},
we update the columns of $\mbfa_n$ from $j=1$ to $R_n$ by 
\begin{align}
\bfa_{n,j}^k=\argmin_{\|\bfa_{n,j}\|_2\le 1}&~
\frac{1}{2}\big\|\bfa_{n,j}\bfb_n^{k,j}+(\tilde{\mbfa}_n^k)_{j^c}(\mbfb_n^k)^{j^c}-\mbfm_{(n)}\big\|_F^2+\lambda_n\|\bfa_{n,j}\|_1\label{an-j}\\[-0.3cm]
&+\mu\left(\big\langle (\tilde{\mbfa}_n^k)_{j^c}(\tilde{\mbfa}_n^k)_{j^c}^\top\hat{\bfa}_{n,j}^k,\bfa_{n,j}-\hat{\bfa}_{n,j}^k\big\rangle+\frac{L_{n,j}^k}{2}\|\bfa_{n,j}-\hat{\bfa}_{n,j}^k\|_2^2\right),\nonumber
\end{align}
where $\bfb_n^{k,j}$ denotes the $j$-th row of $\mbfb_n^k$, $(\mbfb_n^k)^{j^c}$ is the submatrix by taking all rows of $\mbfb_n^k$ except the $j$-th one,  $$\hat{\bfa}_{n,j}^k={\bfa}_{n,j}^{k-1}+\omega_{n,j}^k({\bfa}_{n,j}^{k-1}-{\bfa}_{n,j}^{k-2})$$ is an extrapolated point, $(\tilde{\mbfa}_n^k)_{j^c}$ is short for $\big(\bfa_{n,1}^k,\cdots,\bfa_{n,j-1}^k,\bfa_{n,j+1}^{k-1},\cdots,\bfa_{n,R_n}^{k-1}\big)$, and $L_{n,j}^k$ is a Lipschitz constant of the gradient of $$\frac{1}{2}\left(\sum_{i< j}\big(\bfa_{n,j}^\top\bfa_{n,i}^k\big)^2+\sum_{i> j}\big(\bfa_{n,j}^\top\bfa_{n,i}^{k-1}\big)^2\right)$$ with respect to $\bfa_{n,j}$. 
One can easily write the update in \eqref{an-j} explicitly as
{\small\begin{align}\label{up-an-j}
\bfa_{n,j}^k=\mcp_{B_1}\left[\mcs_{\frac{\lambda_n}{b+\mu L}}\left(\frac{\mu L}{b+\mu L}\hat{\bfa}_{n,j}^k-\frac{\big((\tilde{\mbfa}_n^k)_{j^c}(\mbfb_n^k)^{j^c}-\mbfm_{(n)}\big)\big(\bfb_n^{k,j}\big)^\top}{b+\mu L}-\frac{\mu}{b+\mu L}(\tilde{\mbfa}_n^k)_{j^c}(\tilde{\mbfa}_n^k)_{j^c}^\top\hat{\bfa}_{n,j}^k\right)\right],
\end{align}
}where $b=\|\bfb_n^{k,j}\|_2^2$, $L=L_{n,j}^k$, and $\mcp_{B_1}$ denotes the projection to unit Euclidean ball.

Following the proof of Theorem \ref{thm:convg}, one can show that the method described above has global convergence if the parameters $L_{n,j}^k$, $\omega_{n,j}^k,L_c^{k,n},\omega_c^{k,n}$ satisfy conditions as those in Theorem \ref{thm:convg}. We do not repeat it here.

\section{Numerical experiments}\label{sec:numerical}
In this section, we compare Algorithm \ref{alg:apg} (APG), HONMF in \cite{morup2008algorithms}, and HALS in \cite{phan2011extended} for solving (sparse) NTD on both synthetic and real world data. Also, we test the modified version of Algorithm \ref{alg:apg} and HONMF for solving (sparse) NTD with missing values. The code of all compared solvers is accessible online. There are of course more other solvers for (sparse) NTD such as dGN in \cite{phan2011damped}, ALS in \cite{friedlander2008computing}, and CCD in \cite{liu2011sparse}. However, we do not get the code of dGN, and the code of CCD and ALS only handles the case where the core tensor is fixed to identity tensor. 

All the tests are performed on a laptop with an i7-620m CPU and 3GB RAM and running 32-bit Windows 7 and MATLAB 2010b with Statistics Toolbox and Tensor Toolbox of version 2.5 \cite{TTB_Software}.

\subsection{Implementation details}
This subsection specifies the implementation of Algorithm \ref{alg:apg} in details about initialization and stopping criteria. Unless specified, all parameters for HONMF and HALS are set to their default values.

\subsubsection*{Initialization} For all the compared algorithms, we use the same starting point. Throughout the tests, we first randomly generate ${\mbfa}_1^0,\cdots,{\mbfa}_N^0$ and then process them by the Higher-order Orthogonal Iteration algorithm in \cite{de2000best}. Specifically, for \eqref{eq:spntd}, let 
\begin{equation}\label{eq:init-a}\bm{\mcb}=\bm{\mcm}\times_1(\mbfa_1^0)^\top\cdots\times_{n-1}(\mbfa_{n-1}^0)^\top\times_{n+1}(\mbfa_{n+1}^0)^\top\times_N(\mbfa_{N}^0)^\top,
\end{equation}
and update $\mbfa_n^0=\max(\eps_{machine},\mbfu_n)$ alternatively for $n=1,\cdots,N$, where $\eps_{machine}$ stands for machine precision and $\mbfu_n$ contains the left $R_n$ singular vectors of $\mbfb_{(n)}$. Then set 
\begin{equation}\label{eq:init-c}\bm{\mcc}^0=\bm{\mcm}\times_1(\mbfa_1^0)^\top
\cdots\times_N(\mbfa_{N}^0)^\top.
\end{equation}
For \eqref{eq:spntdc}, we use the same initialization except replacing $\bm{\mcm}$ to $\mcp_\Omega(\bm{\mcm})$ in \eqref{eq:init-a} and \eqref{eq:init-c}. It is observed that all the algorithms perform better with this kind of starting point than a random one, in both convergence speed and chance of avoiding local minima. The use of strictly positive initial points is mainly due to the consideration that HONMF does not allow its iterates to have zero components. 

\subsubsection*{Stopping criteria} We stop Algorithm \ref{alg:apg} and its modified version in section \ref{sec:spntdc} if a maximum number of iterations or maximum time is reached or one of the following conditions is satisfied
\begin{subequations}
\begin{align}
&\frac{\|\mcp_\Omega\big(\bm{\mcc}^k\times_1\mbfa_1^k\cdots\times_N\mbfa_N^k-\bm{\mcm}\big)\|_F}{\|\mcp_\Omega(\bm{\mcm})\|_F}\le tol,\ \text{ for some }k,\\
&\frac{|F_\Omega^k-F_\Omega^{k+1}|}{1+F_\Omega^k}\le tol,\ \text{ for three consecutive }k\text{'s},
\end{align}
\end{subequations}
where $F_\Omega^k\triangleq F_\Omega(\bm{\mcc}^k,\mbfa_1^k,\cdots,\mbfa_N^k)$ and $tol$ is a small positive value specified below. Note that for Algorithm \ref{alg:apg}, $\Omega$ contains all indices.

\subsection{Nonnegative Tucker decomposition}
In this subsection, we compare APG, HONMF, and HALS on solving NTD, i.e., \eqref{eq:spntd} with all of $\lambda_c,\lambda_1,\cdots,\lambda_N$ set to \emph{zero}.
We first test them on two sets of synthetic data and then on two image datasets.

\subsubsection*{Synthetic data}
In the first synthetic dataset, each tensor has the form  $\bm{\mcm}=\bm{\mcc}\times_1\mbfa_1\times_2\mbfa_2\times_3\mbfa_3$, where $\bm{\mcc}$ is generated by MATLAB's command \verb|rand(5,5,5)| and each $\mbfa_i$ by command \verb|max(0,randn(80,5))|. Then $\bm{\mcm}$ is re-scaled to have unit maximum component. Each tensor $\bm{\mcm}$ in the second test is generated in the same way but has an \emph{unbalanced} dimension $10\times10\times1000$, and the core tensor is $3\times3\times30$. We emphasize that uniformly random $\bm{\mcc}$ makes the problem more difficult\footnote{For the case that $\bm{\mcc}$ is also Gaussian randomly generated, the performance of APG and HALS is similar.} than Gaussian random one because the former is not \emph{zero}-mean. The true dimension is used in our tests, namely, $I_n=50,R_n=5,\forall n$ is set in \eqref{eq:spntd} for the first dataset and $(I_1,I_2,I_3)=(10,10,1000),(R_1,R_2,R_3)=(3,3,30)$ for the second one. 

We add normalized noise to each tensor, namely, we input to each algorithm with $\bm{\mcm}^{nois}=\bm{\mcm}+\eta\frac{\|\bm{\mcm}\|_F}{\|\bm{\mcn}\|_F}\bm{\mcn}$, where the entries of $\bm{\mcn}$ follow i.i.d standard Gaussian distribution. We run each algorithm to $t_{\max}$ (sec) and compare their relative error $\frac{\|\bm{\mcc^r}\times_1\mbfa_1^r\times_2\mbfa_2^r\times_3\mbfa_3^r-\bm{\mcm}\|_F}{\|\bm{\mcm}\|_F}$, where $(\bm{\mcc}^r, \mbfa_1^r, \mbfa_2^r, \mbfa_3^r)$ is a solution obtained by running an algorithm. Table \ref{tab:rand-ntd} shows the average relative error and number of iterations for the three algorithms over 20 independent runs with $t_{\max}=10$ and different $\eta$'s.  Figure \ref{fig:rand-ntd} plots how the relative error changes with respect to the running time for each algorithm with $t_{\max}=20$ and also to iterations.

From the table, we see that APG performs significantly better than HONMF and HALS for noiseless case. When there is noise, i.e., $\eta>0$, APG is still much better than HONMF and comparable to HALS. From the figure, we see that HONMF converges very slowly\footnote{The code of HONMF is implemented for NTD with missing value. Its running time would be reduced if it were implemented separately for the NTD. However, we observe that HONMF converges much slower than our algorithm.} in both cases and HALS works well for $\bm{\mcm}$ with balanced dimension but converges slowly for the unbalanced one. APG converges faster than both HONMF and HALS, in particular for the unbalanced case. 

To see how the algorithms perform on decomposing nonnegative tensors with larger ranks, we also test them on random tensors generated in the same way as above with size $80\times 80\times 80$ and each mode rank $r$, where $r$ varies from 3 to 30 with increment 3. Each algorithm runs to 1,000 iterations. Figure \ref{fig:rand-diff-rank} plots the average relative errors of 10 independent runs for each algorithm. From the figure, we see that APG performs consistently better than HONMF and HALS and much better when $r$ is small.

\begin{table}\caption{Average results over 20 independent runs by APG, HONMF and HALS on two synthetic datasets}\label{tab:rand-ntd}
\begin{center}
{\footnotesize
\begin{tabular}{|c|cc|cc|cc|}\hline
& \multicolumn{2}{|c|}{APG} & \multicolumn{2}{|c|}{HONMF} & \multicolumn{2}{|c|}{HALS}\\\hline
noise level & rel. err. & \# iter & rel. err. & \# iter & rel. err. & \# iter\\\hline
\multicolumn{7}{|c|}{$(I_1,I_2,I_3)=(80,80,80), (R_1,R_2,R_3)=(5,5,5)$}\\\hline
$\eta = 0.00$ & 7.09e-004 & 467 & 6.06e-002 & 87 & 2.45e-003 & 758\\
$\eta = 0.05$ & 2.79e-003 & 468 & 6.86e-002 & 48 & 3.27e-003 & 732\\
$\eta = 0.10$ & 4.78e-003 & 466 & 7.15e-002 & 47 & 5.44e-003 & 759\\\hline\hline
\multicolumn{7}{|c|}{$(I_1,I_2,I_3)=(10,10,1000), (R_1,R_2,R_3)=(3,3,30)$}\\\hline
$\eta = 0.00$ & 5.12e-004 & 653 & 2.52e-002 & 287 & 2.97e-003 & 737\\
$\eta = 0.05$ & 1.49e-002 & 668 & 3.00e-002 & 232 & 1.50e-002 & 739\\
$\eta = 0.10$ & 3.02e-002 & 670 & 3.84e-002 & 222 & 3.01e-002 & 740\\\hline
\end{tabular}}
\end{center}
\end{table}

\begin{figure}\caption{Convergence behavior of APG, HONMF and HALS on synthetic data. Left: $80\times80\times80$ nonnegative tensor $\bm{\mcm}$ and $5\times5\times5$ core tensor $\bm{\mcc}$; Right: $10\times10\times1000$ nonnegative tensor $\bm{\mcm}$ and $3\times3\times30$ core tensor $\bm{\mcc}$.}
\label{fig:rand-ntd}
\begin{center}
\includegraphics[width=0.35\textwidth]{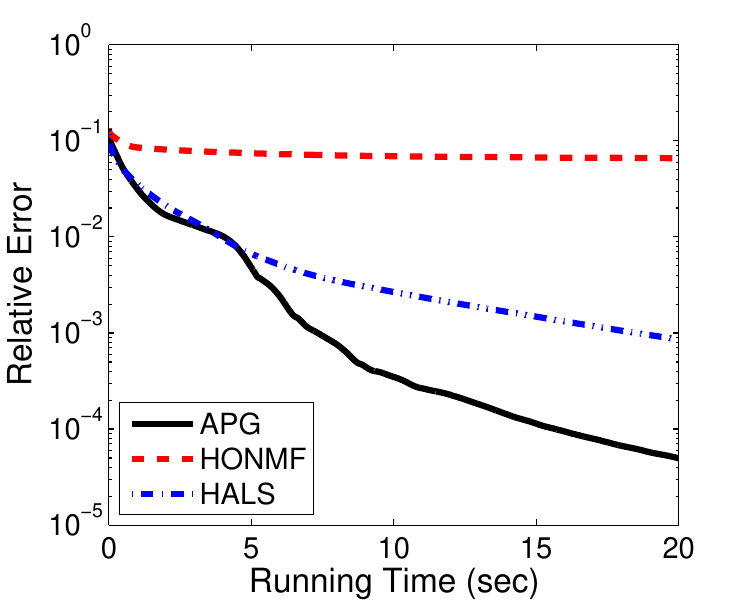}
\includegraphics[width=0.35\textwidth]{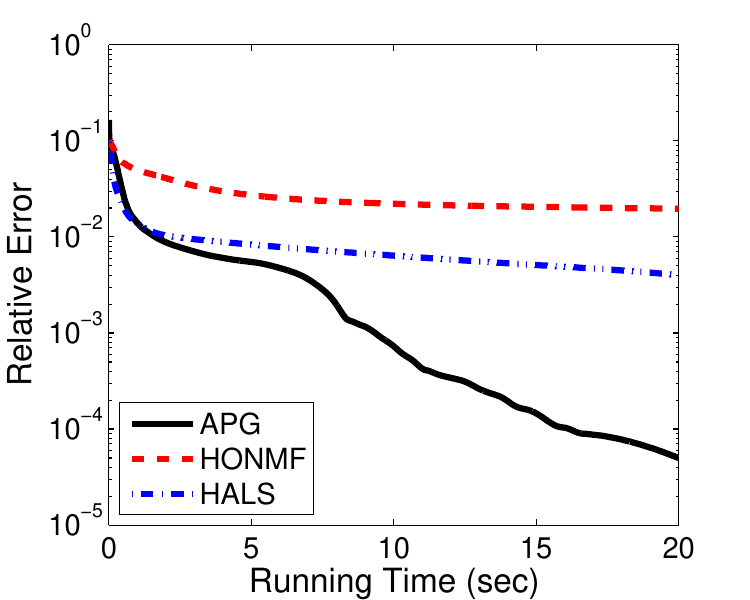}\\
\includegraphics[width=0.35\textwidth]{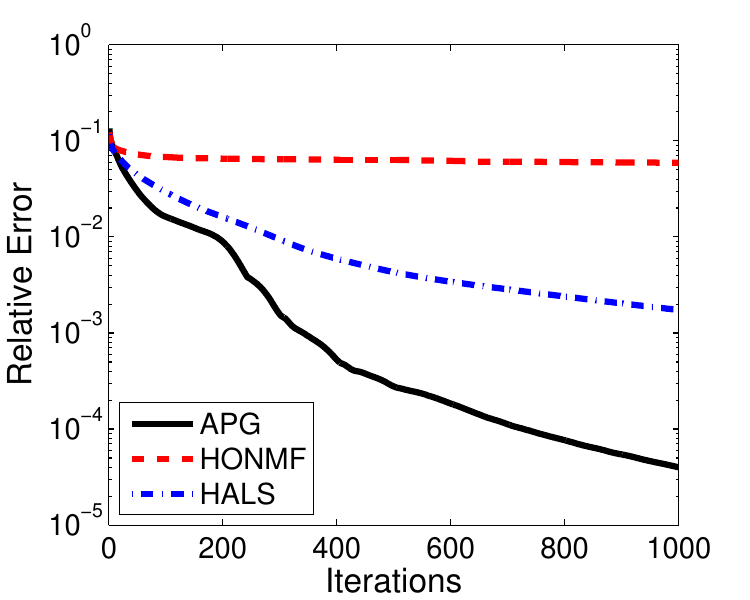}
\includegraphics[width=0.35\textwidth]{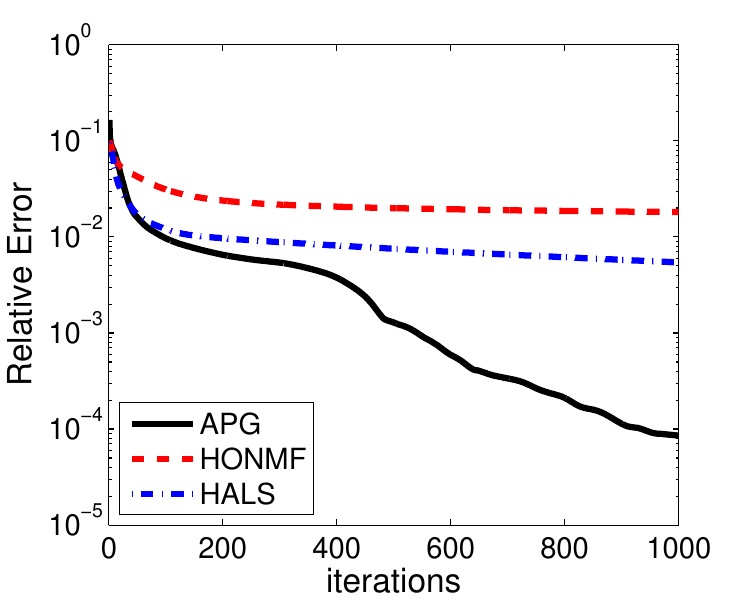}
\end{center}
\end{figure}

\subsubsection*{Image data}
The first test uses the Swimmer dataset constructed in \cite{donoho2003}, which has 256 swimmer images and each one has resolution of $32\times32$. We form a $32\times32\times256$ tensor $\bm{\mcm}$ using the dataset and then re-scale it to have unit maximum component. The core dimension is set to $(24,20,20)$\footnote{The mode-$n$ ranks of  $\bm{\mcm}$ are 24, 14, and 13 for $n=1,2,3$, respectively. Larger size is used to improve the data fitting.}. We run APG, HONMF, and HALS to $t_{\max}=30$ (sec) and plot their relative errors on the left of Figure \ref{fig:image-ntd}. The second test uses a brain MRI image of size $181\times 217\times181$, which has been tested in \cite{liu2011sparse} for sparse nonnegative tensor decomposition. We re-scale it to have unit maximum pixel and set the core size to $(30,30,30)$. All the three algorithms run to $t_{\max}=600$ (sec), and the relative errors are plotted on the right of Figure \ref{fig:image-ntd}. From the figure, we see that HONMF performs the worst and HALS decreases the objective faster than APG in the beginning but APG eventually converges faster. In particular for the test with Swimmer dataset, the overall convergence speed of APG is much faster than that of HALS, and APG reaches much lower relative errors while HALS seems to be trapped at some local solution\footnote{Sometimes, APG is also  trapped at some local solution. We run the three algorithms on the Swimmer dataset to maximum 30 seconds. If the relative error is below $10^{-3}$, we regard the algorithm reaches a global solution. Among 20 independent runs, APG, HONMF, and HALS reach a global solution 11, 0, and 5 times, respectively. We also test the three algorithms with smaller rank (24,18,17), in which case APG, HONMF, and HALS reach a global solution 16, 0, and 4 times respectively among 20 independent runs.}.

\begin{figure}\caption{Average relative errors of 10 independent runs for APG, HONMF, and HALS on synthetic tensors of size $80\times80\times80$ and with each mode rank $r$.}
\label{fig:rand-diff-rank}
\begin{center}
\includegraphics[width=0.35\textwidth]{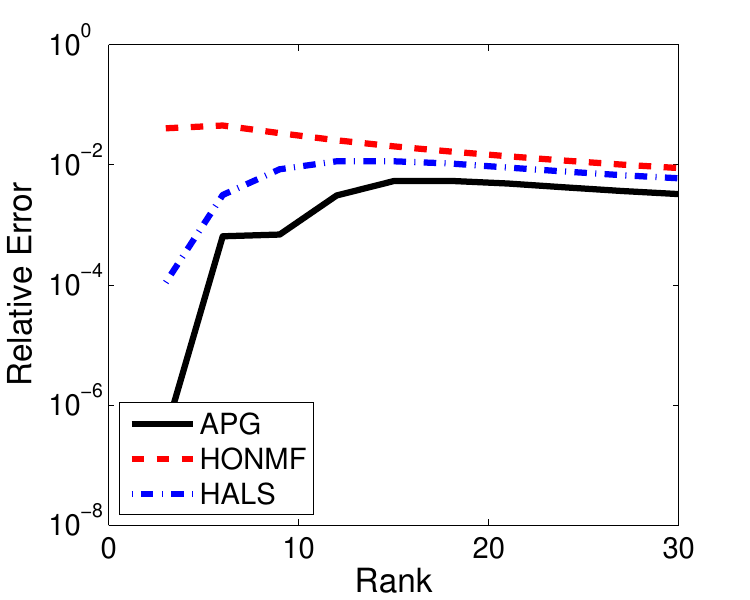}
\end{center}
\end{figure}

\begin{figure}\caption{Convergence behavior of APG, HONMF, and HALS on Swimmer dataset (left) and a brain MRI image (right).}
\label{fig:image-ntd}
\begin{center}
\includegraphics[width=0.35\textwidth]{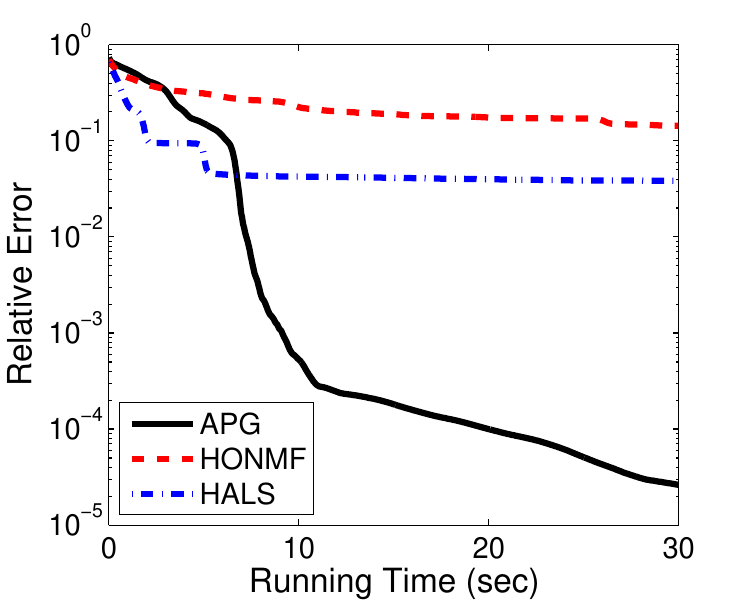}
\includegraphics[width=0.35\textwidth]{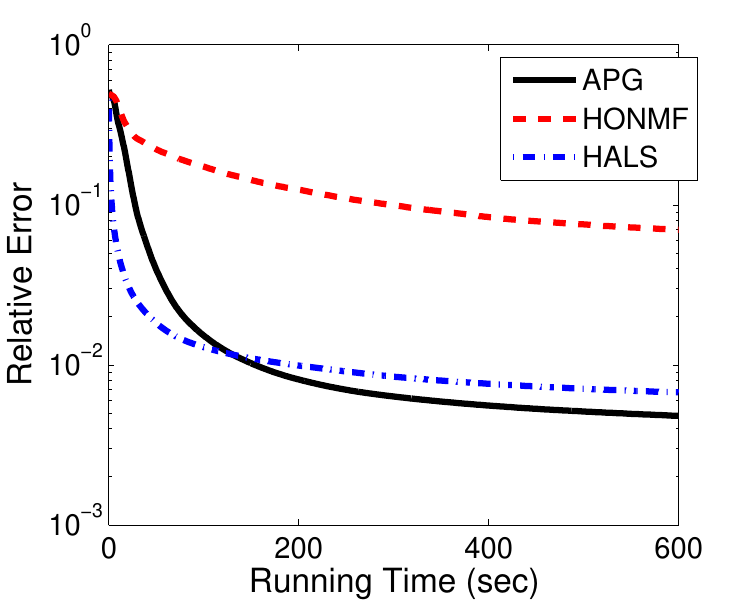}
\end{center}
\end{figure}

\subsection{Sparse nonnegative Tucker decomposition}
In this subsection, we compare APG and HONMF for solving sparse NTD, i.e., \eqref{eq:spntd} with at least one of $\lambda_c,\lambda_1,\cdots,\lambda_N$ set to be positive.
 HALS is not coded\footnote{In the implementation of HALS, all factor matrices are re-scaled such that each column has unit length after each iteration. The re-scaling is necessary for efficient update of the core tensor and does not change the objective value of \eqref{eq:spntd} if all  sparsity paramenters are \emph{zero}. However, it will change the objective if some of $\lambda_c,\lambda_1,\cdots,\lambda_N$ are positive.} for sparse NTD.  Hence, we do not include HALS for comparison. 

We compare APG and HONMF on the brain MRI image used above and the CBCL face image dataset\footnote{http://www.ai.mit.edu/projects/cbcl} which has been tested in \cite{shashua2005non} for nonnegative tensor decomposition. For the brain MRI image, we set $R_1=R_2=R_3=30$ and $\lambda_c=\lambda_1=\lambda_2=\lambda_3=0.5$ in \eqref{eq:spntd}. We run APG and HONMF to $t_{\max}=300$ (sec) and report the results at time $t=100,200,300$ (sec). Table \ref{tab:MRI-ntds} summarizes the average results of 10 independent runs. The ``core den.'' is calculated by $\frac{\#\text{ nonzeros of }\bm{\mcc}^r}{30^3}$ and ``fac. den.'' by $\frac{\sum_{n=1}^3\#\text{ nonzeros of }\mbfa_n^r}{30\cdot(181+217+181)}$. 
We see that APG reaches much lower objective values and relative errors than those by HONMF. In addition, the solutions obtained by APG are sparser than those by HONMF and are potentially easier to interpret.

The CBCL dataset has 6977 face images, and each one is $19\times 19$. We use all these images to form a $19\times19\times6977$ nonnegative tensor $\bm{\mcm}$, which is then re-scaled to have unit maximum component. The core size is set to $(R_1,R_2,R_3)=(5,5,50)$ and the sparsity parameters to $\lambda_c=0.5,\lambda_1=\lambda_2=\lambda_3=0$, namely, we only want the core tensor to be sparse. Table \ref{tab:cbcl-ntds} reports the average results obtained by APG and HONMF at running time $t=25,50,75,100$ (sec). We see that APG reaches much lower objective values and also lower relative errors than those by HONMF. The solutions given by APG are much sparser than those by HONMF. This may be because APG uses the constraints \eqref{eq:con-a} while HONMF simply normalizes each factor matrix after every iteration. However, it somehow validates the use of the constraints \eqref{eq:con-a}.

\begin{table}\caption{Average results by APG and HONMF on a brain MRI image with the core size $R_1=R_2=R_3=30$}\label{tab:MRI-ntds}
{\small\begin{center}
\resizebox{\textwidth}{!}{\begin{tabular}{|c|ccccc|ccccc|}
\hline
 &\multicolumn{5}{|c|}{APG} & \multicolumn{5}{|c|}{HONMF}\\\hline
time & obj. & rel. err. & fac. den. & core den. & \# iter & obj. & rel. err. & fac. den. & core den. & \# iter \\\hline
100 & 1.6622e+3 & 6.15e-2 & 32.45\% &  7.62\% & 185 & 6.2934e+3 & 1.77e-1 & 32.47\% & 31.84\% & 31\\
200 & 8.4659e+2 & 2.94e-2 & 23.22\% & 14.45\% & 370 & 4.7762e+3 & 1.48e-1 & 30.61\% & 31.01\% & 48\\
300 & 6.8898e+2 & 2.25e-2 & 20.49\% & 16.87\% & 555 & 4.0240e+3 & 1.30e-1 & 29.14\% & 29.91\% & 63\\\hline
\end{tabular}}
\end{center}
}
\end{table}

\begin{table}\caption{Average results by APG and HONMF on CBCL dataset with the core size $(R_1,R_2,R_3)=(5,5,50)$}\label{tab:cbcl-ntds}
{\footnotesize\begin{center}
\resizebox{\textwidth}{!}{\begin{tabular}{|c|cccc|cccc|}
\hline
 & \multicolumn{4}{|c|}{APG} & \multicolumn{4}{|c|}{HONMF}\\\hline
time &  obj. & rel. err. & core den. & \# iter & obj. & rel. err.  & core den. & \# iter\\\hline
25 & 3.2469e+4 & 2.72e-1 & 11.45\% & 135 & 5.9824e+4 & 3.63e-1 & 90.19\% & 29\\
50 & 3.1453e+4 & 2.68e-1 & 7.56\% & 271 & 5.3017e+4 & 3.40e-1 & 69.26\% & 57\\
75 & 3.1370e+4 & 2.68e-1 & 6.78\% & 408 & 4.9786e+4 & 3.28e-1 & 58.03\% & 84\\
100 & 3.1344e+4 & 2.67e-1 & 6.46\% & 545 & 4.7289e+4 & 3.20e-1 & 51.26\% & 112\\\hline
\end{tabular}}
\end{center}
}
\end{table}

\subsection{Sparse nonnegative Tucker decomposition with missing values}
In this subsection, we test APG for solving \eqref{eq:spntdc} on synthetic data and compare it to HONMF on the brain MRI image used above.

\subsubsection*{Performance of APG with different sample ratios} First, we show that APG using partial observations can achieve similar accuracies as that using full observations. Each tensor has the form $\bm{\mcm}=\bm{\mcc}\times_1\mbfa_1\times_2\mbfa_2\times_3\mbfa_3$ and is re-scaled to have unit maximum component, where $\bm{\mcc}$ is generated by MATLAB's command \verb|max(0,randn(R,R,R))| and each factor matrix $\mbfa_n$ by \verb|max(0,randn(50,R))| with \verb|R| varying among $\{5,8,11,14,17,20,23,26\}$. We choose $\text{SR}=10\%,30\%,50\%,100\%$ samples uniformly at random and compare the performance of APG using different SRs. The maximum number of iterations is set to 5,000 and the stopping tolerance to $tol=10^{-5}$. Figure \ref{fig:rand-ntdc} plots the average relative errors and running time (sec) of APG over 20 independent trials. We see that APG using 30\% and 50\% samples gives similar accuracies as that using full observations. APG with 10\% samples can still make relative errors low to about 1\%  as $R\le 14$, but $10\%$ samples seem not enough when $R\ge 17$. Longer time by APG with partial observations is due to the extra update \eqref{update-x} and more iterations. When $R\ge17$, the running time of APG with 10\% samples decreases because it stops earlier.

\begin{figure}\caption{Average relative errors (left) and runnting time (right) by APG using different sample ratios}
\label{fig:rand-ntdc}
\begin{center}
\includegraphics[width=0.35\textwidth]{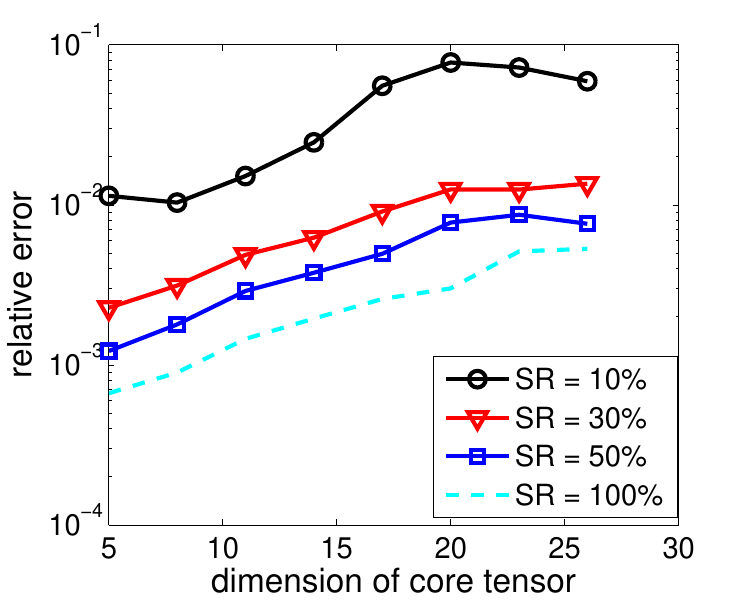}
\includegraphics[width=0.35\textwidth]{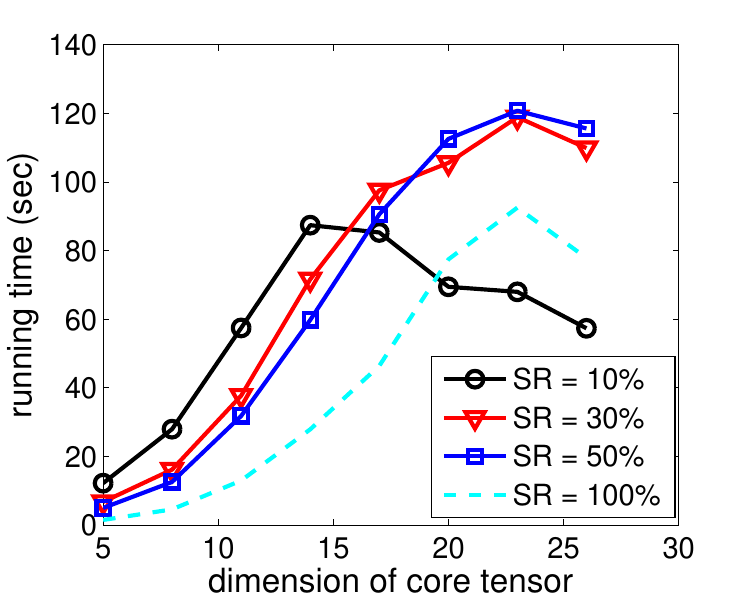}
\end{center}
\end{figure}

\subsubsection*{Comparison with HONMF\footnote{Although HONMF converges very slowly, it is the only one we can find that is also coded for sparse nonnegative Tucker decomposition with missing values.}} Secondly, we compare APG to HONMF on the brain MRI image used above. The core dimension is set to $R_1=R_2=R_3=30$ and sparsity parameters to $\lambda_c=\lambda_1=\lambda_2=\lambda_3=0.5$. We compare the two algorithms using $\text{SR}=10\%,30\%,50\%$ uniformly randomly chosen samples and run them to $t_{\max}=600$ (sec). Table \ref{tab:MRI-ntdc} shows the average results at time $t=150,300,450,600$ for different SRs over 5 independent trials. From the table, we see that HONMF fails with 10\% samples while APG can still work reasonably. In all cases, APG performs better than HONMF in both accuracy and speed. The solutions given by APG are sparser than those by HONMF for $\text{SR}=30\%,50\%$.

\begin{table}\caption{Average results by APG and HONMF on a brain MRI image from different samples}\label{tab:MRI-ntdc}
{\small\begin{center}
\resizebox{\textwidth}{!}{\begin{tabular}{|c|ccccc|ccccc|}
\hline
& \multicolumn{5}{|c|}{APG} & \multicolumn{5}{|c|}{HONMF}\\\hline
time & obj. & rel. err. & fac. den. & core den. & \# iter & obj. & rel. err. & fac. den. & core den. & \# iter \\\hline\hline
\multicolumn{11}{|c|}{SR = 10\%}\\\hline
150 & 1.3418e+3 & 2.32e-1 & 30.49\% & 0.72\% & 208 & 1.5608e+4 & 1.00e+0 & 0.00\% & 0.00\% & 67\\
300 & 9.3130e+2 & 1.80e-1 & 17.03\% & 0.88\% & 416 & 1.5608e+4 & 1.00e+0 & 0.00\% & 0.00\% & 150\\
450 & 7.9761e+2 & 1.60e-1 & 14.28\% & 1.23\% & 623 & 1.5608e+4 & 1.00e+0 & 0.00\% & 0.00\% & 232\\
600 & 7.4748e+2 & 1.53e-1 & 13.60\% & 1.40\% & 831 & 1.5608e+4 & 1.00e+0 & 0.00\% & 0.00\% & 315\\\hline
\multicolumn{11}{|c|}{SR = 30\%}\\\hline
150 & 1.5808e+3 & 1.27e-1 & 35.15\% & 2.31\% & 191 & 2.4525e+3 & 1.88e-1 & 31.80\% & 41.87\% & 40\\
300 & 9.0284e+2 & 8.17e-2 & 19.65\% & 4.38\% & 384 & 1.9277e+3 & 1.58e-1 & 28.90\% & 38.48\% & 67\\
450 & 7.0151e+2 & 6.25e-2 & 17.29\% & 6.34\% & 576 & 1.6362e+3 & 1.38e-1 & 26.44\% & 33.07\% & 96\\
600 & 6.2076e+2 & 5.34e-2 & 15.69\% & 7.60\% & 769 & 1.4587e+3 & 1.25e-1 & 24.54\% & 30.12\% & 129\\\hline
\multicolumn{11}{|c|}{SR = 50\%}\\\hline
150 & 1.8767e+3 & 1.08e-1 & 35.03\% & 3.64\% & 184 & 3.7494e+3 & 1.91e-1 & 31.75\% & 36.66\% & 40\\
300 & 9.5363e+2 & 5.88e-2 & 22.42\% & 7.55\% & 367 & 2.8367e+3 & 1.59e-1 & 28.70\% & 40.15\% & 64\\
450 & 7.0877e+2 & 4.03e-2 & 19.29\% & 10.74\% & 550 & 2.3369e+3 & 1.37e-1 & 26.48\% & 42.33\% & 92\\
600 & 6.2737e+2 & 3.32e-2 & 17.75\% & 12.41\% & 733 & 2.0265e+3 & 1.23e-1 & 24.96\% & 42.58\% & 124\\\hline
\end{tabular}}
\end{center}
}
\end{table}

\subsection{Sparse higher-order principal component analysis}
We use a simple test with synthetic data to show that \eqref{eq:rhopca} can be better than unregularized HOPCA that sets all of $\lambda_c,\lambda_1,\cdots,\lambda_N$ to \emph{zero} in \eqref{eq:hopca}. We use the APG method described in Section \ref{sec:hopca} for \eqref{eq:rhopca} and HOOI \cite{de2000best} for the unregularized HOPCA. We set $L_{n,j}^k=\|(\tilde{\mbfa}_n^k)_{j^c}(\tilde{\mbfa}_n^k)_{j^c}^\top\|$ in \eqref{up-an-j} and $\omega_{n,j}^k$ in the same way as in \eqref{eq:weighta}.

We generate a $50\times50\times50$ tensor in the form of $\bm{\mcm}=\bm{\mcc}\times_1\mbfa_1\times_2\mbfa_2\times\mbfa_3+\bm{\mcn}$. Here, $\bm{\mcc}$ is $3\times3\times3$, and each element is drawn from standard Gaussian distribution. Then 60\% components of $\bm{\mcc}$ are selected uniformly at random and set to \emph{zero}. Factor matrices have sparsity patterns shown in Figure \ref{fig:pattern}, and each non-zero element is drawn from standard Gaussian distribution. Then each column is normalized. $\bm{\mcn}$ is Gaussian random noise and makes the signal-to-noise-ratio $\text{SNR}=60$. The sparsity parameters are set to $\lambda_c=\lambda_1=\lambda_2=\lambda_3=0.02$, and orthogonality parameter is tuned to $\mu=0.1$ in \eqref{eq:rhopca}. The sparsity patterns of the original $\bm{\mcc}$ and $\mbfa$ and those\footnote{We permute the columns of the factor matrices and do permutations to the core tensor accordingly.} given by APG are plotted in Figure \ref{fig:pattern}. We see that the solution given by APG have almost the same sparsity pattern as the original ones. To see how close to orthogonality each factor matrix is given by APG, we first normalize each column of the factor matrices and then calculate $\|\mbfa_n^\top\mbfa_n-\mbfi\|_F/\|\mbfi\|_F$, which are $2.95\times10^{-3}, 1.36\times10^{-3}, 7.24\times10^{-5}$, respectively for $n=1,2,3$. Hence, they are almost orthogonal. Although the solution by HOOI makes a relatively higher data fitting,  it is highly dense with no zero element.  Therefore, the relaxed model \eqref{eq:rhopca} can potentially give better solution than \eqref{eq:hopca} for some applications such as classification.

\begin{figure}\caption{Sparsity pattern of the orginal $\bm{\mcc}$ and $\mbfa$ and those given by APG method}\label{fig:pattern}
\centering
\begin{tabular}{cc}\\
Original $\mbfa$ & APG's $\mbfa$\\
\includegraphics[width = 0.12\textwidth]{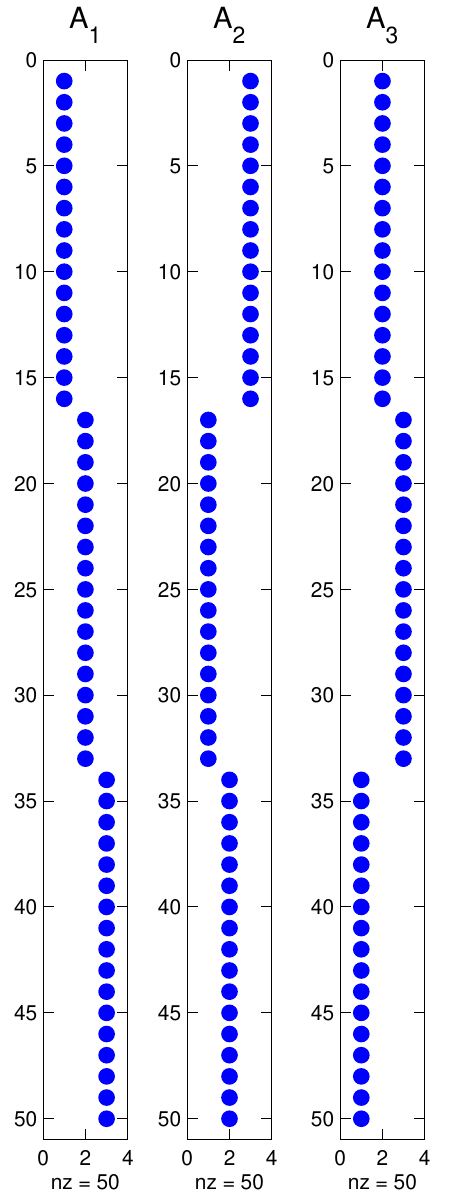}&\includegraphics[width = 0.12\textwidth]{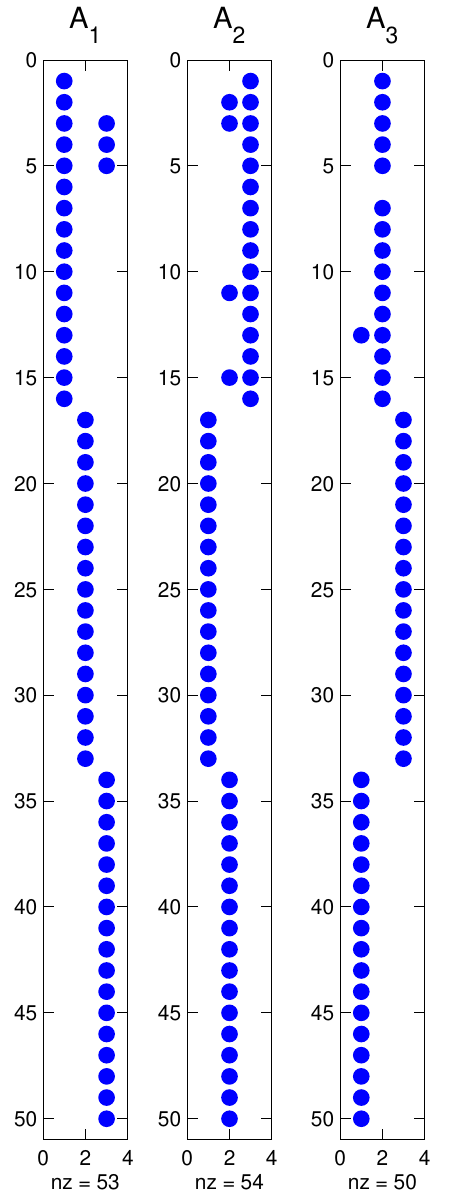}
\end{tabular}
\begin{tabular}{c}
Original $\bm{\mcc}$\\
\includegraphics[width = 0.32\textwidth]{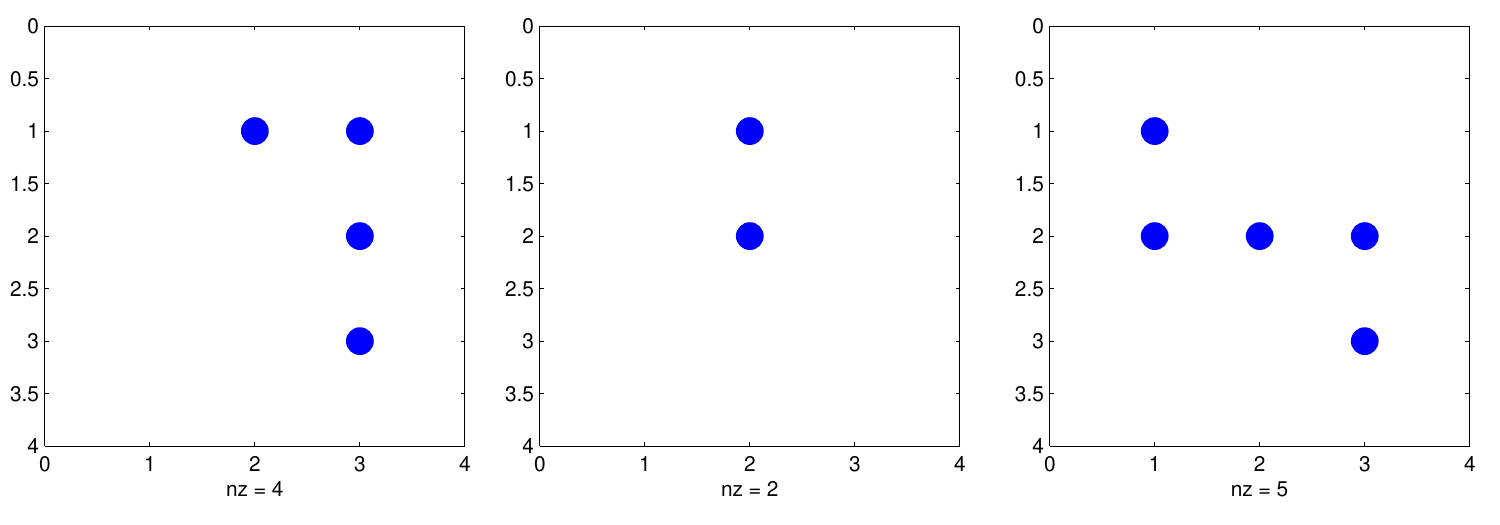}\\
APG's $\bm{\mcc}$\\
\includegraphics[width = 0.32\textwidth]{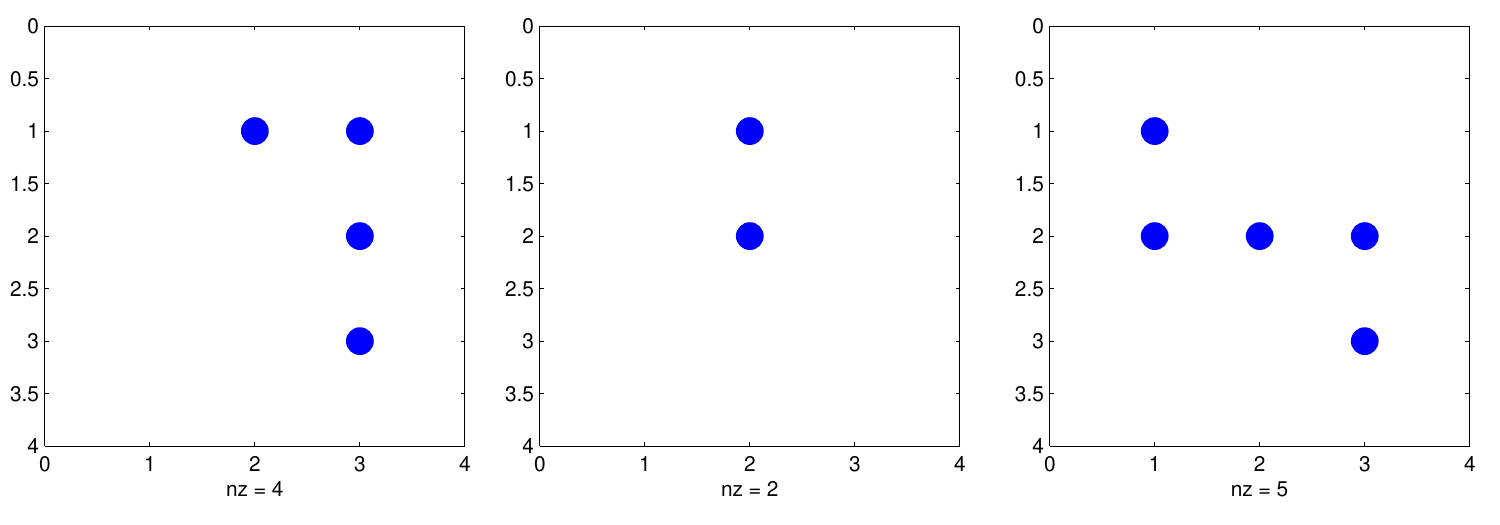}
\end{tabular}
\end{figure} 

\section{Conclusions}\label{sec:discussion} Sparse NTD aims at decomposing a tensor into the product of a core tensor and some factor matrices with nonnegativity and sparsity constraints. 
Existing algorithms for this problem either converge rapidly with very expensive per-iteration cost or have low per-iteration cost with very slow convergence speed. We have proposed the APG method, which owns both low per-iteration complexity and fast convergence speed. Moreover, the algorithm has been modified for sparse NTD from partial observations of a target tensor. The modified algorithm also has low per-iteration cost and can give similar decompositions from half of or even fewer observations as those from full observations.

\section*{Acknowledgements}
This work is partly supported by ARL and ARO grant W911NF-09-1-0383 and AFOSR FA9550-10-C-0108. The author would like to thank three anonymous referees, the technical editor and the associate editor for their very valuable comments and suggestions. Also, the author would like to thank Prof. Wotao Yin for his valuable discussions and Anh Huy Phan for sharing the code of HALS.

\appendix
\section{Efficient computation}
The most expensive step in Algorithm \ref{alg:apg} is the computation of $\nabla_{\bm{\mcc}}\ell(\bm{\mcc},
\mbfa)$ and $\nabla_{\mbfa_n}\ell(\bm{\mcc},
{\mbfa})$ in \eqref{update-c} and \eqref{update-a}, respectively. Note that we have omitted the superscript. Next, we discuss how to efficiently compute them.

\subsubsection*{Computation of $\nabla_{\bm{\mcc}}\ell$}
According to \eqref{eq:vec}, we have
$$\ell(\bm{\mcc},\mbfa)=\frac{1}{2}\big\|\big(\otimes_{n=N}^1\mbfa_n\big)\vc(\bm{\mcc})-\vc(\bm{\mcm})\big\|_2^2.$$
Using the properties of Kronecker product (see \cite{horn1991topics}, for example), 
we have
\begin{equation}\label{eq:vecgrad}
\vc\big(\nabla_{\bm{\mcc}}\ell(\bm{\mcc},\mbfa)\big)=\big(\otimes_{n=N}^1\mbfa_n^\top \mbfa_n\big)\vc(\bm{\mcc})-\big(\otimes_{n=N}^1\mbfa_n^\top \big)\vc(\bm{\mcm}).
\end{equation}
It is extremely expensive to explicitly reformulate the Kronecker products in \eqref{eq:vecgrad}. Fortunately, we can use \eqref{eq:vec} again to have
$$\big(\otimes_{n=N}^1\mbfa_n^\top \mbfa_n\big)\vc(\bm{\mcc})=\vc\big(\bm{\mcc}\times_1\mbfa_1^\top \mbfa_1\cdots\times_N \mbfa_N^\top \mbfa_N\big)$$
and
$$\big(\otimes_{n=N}^1\mbfa_n^\top \big)\vc(\bm{\mcm})=\vc\big(\bm{\mcm}\times_1\mbfa_1^\top\cdots\times_N\mbfa_N^\top\big).$$
Hence, we have from \eqref{eq:vecgrad} and the above two equalities that
\begin{equation}\label{eq:gradc} 
\nabla_{\bm{\mcc}}\ell(\bm{\mcc},\mbfa)=\bm{\mcc}\times_1\mbfa_1^\top \mbfa_1\cdots\times_N \mbfa_N^\top \mbfa_N-\bm{\mcm}\times_1\mbfa_1^\top\cdots\times_N\mbfa_N^\top.
\end{equation}

\subsubsection*{Computation of $\nabla_{\mbfa_n}\ell$}
According to \eqref{eq:mat}, we have
\begin{equation}\label{eq-ell}\ell(\bm{\mcc},\mbfa)=\frac{1}{2}\big\|\mbfa_n\mbfc_{(n)}\big(\otimes_{\substack{i=N\\i\neq n}}^1\mbfa_i\big)^\top-\mbfm_{(n)}\big\|_F^2.
\end{equation}
Hence,
\begin{equation}\label{eq:grada}\nabla_{\mbfa_n}\ell(\bm{\mcc},\mbfa)=\mbfa_n(\mbfb_n\mbfb_n^\top)-\mbfm_{(n)}\mbfb_n^\top
\end{equation}
where 
\begin{equation}\label{eq:b}\mbfb_n=\mbfc_{(n)}\big(\otimes_{\substack{i=N\\i\neq n}}^1\mbfa_i\big)^\top.
\end{equation} Similar to what has been done to \eqref{eq:vecgrad}, we do not explicitly reformulate the Kronecker product in \eqref{eq:b} but let
\begin{equation}\label{eq:compb}
\bm{\mcx}=\bm{\mcc}\times_1\mbfa_1\cdots\times_{n-1}\mbfa_{n-1}\times_{n+1}
\mbfa_{n+1}\cdots\times_N\mbfa_N.
\end{equation}
Then we have $\mbfb_n=\mbfx_{(n)}$ according to \eqref{eq:mat}.

\section{Complexity analysis of Algorithm \ref{alg:apg}}\label{sec:complexity}
The main cost of Algorithm \ref{alg:apg} lies in computing $\nabla_{\bm{\mcc}}\ell(\bm{\mcc},\mbfa)$ and $\nabla_{\mbfa_n}\ell(\bm{\mcc}, \mbfa)$, which are required in \eqref{update-c} and \eqref{update-a}, respectively. Note that we have omitted all superscripts for simplicity. Through \eqref{eq:gradc}, the computation of $\nabla_{\bm{\mcc}}\ell(\bm{\mcc},\mbfa)$ requires 
\begin{equation}\label{eq:cost-c}
C\left(\sum_{j=1}^N R_j^2 I_j+\sum_{j=1}^NR_j\prod_{i=1}^N R_i+\sum_{j=1}^N\big(\prod_{i=1}^j R_i\big)\big(\prod_{i=j}^N I_i\big)\right)
\end{equation}
flops, where $C\approx 2$, the first part comes from the computation of all $\mbfa_i^\top\mbfa_i$'s, and the second and third parts are respectively from the computations of the first and second terms in \eqref{eq:gradc}. Disregarding\footnote{In tensor-matrix multiplications, unfolding and folding a tensor both happens, and they can take about a half of time in the whole process of tensor-matrix multiplication. The readers can refer to \cite{schatz2013exploiting} for issues about the cost of tensor unfolding and permutation.} the time for unfolding a tensor and using \eqref{eq:grada}, we have the cost for $\nabla_{\mbfa_n}\ell(\bm{\mcc},\mbfa)$ to be
{\small\begin{align}\label{eq:cost-a}
&C\left(\underset{\text{part 1}}{\underbrace{\sum_{j=1}^{n-1}\big(\prod_{i=1}^jI_i\big)\big(\prod_{i=j}^N R_i\big)
+R_n\big(\prod_{i=1}^{n-1}I_i\big)\sum_{j=n+1}^N\big(\prod_{i=n+1}^jI_i\big)
\big(\prod_{i=j}^NR_i\big)}}\right.\cr
&\hspace{1cm}\left.+\underset{\text{part 2}}{\underbrace{R_n^2\prod_{i\neq n}I_i+R_n^2I_n}}+\underset{\text{part 3}}{\underbrace{R_n\prod_{i=1}^NI_i}}\right),
\end{align}
}where $C$ is the same as that in \eqref{eq:cost-c}, ``part 1'' is for the computation of  $\mbfb_n$ via \eqref{eq:compb}, ``part 2'' and ``part 3'' are respectively from the computations of the first and second terms in \eqref{eq:grada}.

Suppose $R_i<I_i$ for all $i=1,\cdots,N$. Then the quantity of \eqref{eq:cost-c} is dominated by the third part because in this case, 
$$R_j^2I_j<\big(\prod_{i=1}^j R_i\big)\big(\prod_{i=j}^N I_i\big),\qquad R_j\prod_{i=1}^N R_i<\big(\prod_{i=1}^j R_i\big)\big(\prod_{i=j}^N I_i\big).$$
The quantity of \eqref{eq:cost-a} is dominated by the first and third parts. Only taking account of the dominating terms, we claim that the quantities of \eqref{eq:cost-c} and \eqref{eq:cost-a} are similar. To see this, assume $R_i=R, I_i=I,$ for all $i$'s. Then the third part of \eqref{eq:cost-c} is $\sum_{j=1}^NR^jI^{N-j+1}$, and the sum of the first and third parts of \eqref{eq:cost-a} is
\begin{align*}
&\sum_{j=1}^{n-1}\big(\prod_{i=1}^jI_i\big)\big(\prod_{i=j}^N R_i\big)
+R_n\big(\prod_{i=1}^{n-1}I_i\big)\sum_{j=n+1}^N\big(\prod_{i=n+1}^jI_i\big)
\big(\prod_{i=j}^NR_i\big)+R_n\prod_{i=1}^NI_i\\
=&\sum_{j=1}^{n-1}I^jR^{N-j+1}+\sum_{j=n+1}^N I^{j-1}R^{N-j+2}+RI^N\\
=&\sum_{j=N-n+2}^NR^jI^{N-j+1}+\sum_{j=2}^{N-n+1}R^jI^{N-j+1}+RI^N\\
=&\sum_{j=1}^NR^jI^{N-j+1}.
\end{align*}
Hence, the costs for computing $\nabla_{\bm{\mcc}}\ell(\bm{\mcc},\mbfa)$ and $\nabla_{\mbfa_n}\ell(\bm{\mcc},\mbfa)$ are similar.

After obtaining the partial gradients $\nabla_{\bm{\mcc}}\ell(\bm{\mcc},\mbfa)$ and $\nabla_{\mbfa_n}\ell(\bm{\mcc},\mbfa)$, it remains to do some projections to nonnegative orthant to finish the updates in \eqref{update-c} and \eqref{update-a}, and the cost is proportional to the size of $\bm{\mcc}$ and $\mbfa_n$, i.e., $C_p\prod_{i=1}^NR_i$ and $C_pI_nR_n$ with $C_p\approx4$. The data fitting term can be evaluated by
$$\ell(\bm{\mcc},\mbfa)=\frac{1}{2}\left(\langle\mbfa_n^\top\mbfa_n,\mbfb_n\mbfb_n^\top\rangle-2\langle\mbfa_n,\mbfm_{(n)}\mbfb_n^\top\rangle+\|\bm{\mcm}\|_F^2\right),$$
where $\mbfb_n$ is defined in \eqref{eq:b}. Note that $\mbfa_n^\top\mbfa_n$, $\mbfb_n\mbfb_n^\top$ and $\mbfm_{(n)}\mbfb_n^\top$ have been obtained during the computation of $\nabla_{\bm{\mcc}}\ell(\bm{\mcc},\mbfa)$ and $\nabla_{\mbfa_n}\ell(\bm{\mcc},\mbfa)$, and $\|\bm{\mcm}\|_F^2$ can be pre-computed before running the algorithm. Hence, we need $C(R_n^2+I_nR_n)$ additional flops to evaluate $\ell(\bm{\mcc},\mbfa)$, where $C\approx2$. To get the objective value, we need $C(\prod_{i=1}^NR_i+\sum_{i=1}^NI_iR_i)$ more flops for the regularization terms.

Some more computations occur in choosing Lipschitz constants $L_c$ and $L_n$'s. When $R_n\ll I_n$ for all $n$, the cost for computing Lipschitz constants, projection to nonnegative orthant and objective evaluation is negligible compared to that for computing partial gradients $\nabla_{\bm{\mcc}}\ell(\bm{\mcc},\mbfa)$ and $\nabla_{\mbfa_n}\ell(\bm{\mcc},\mbfa)$. Omitting the negligible cost and only accounting the main cost in \eqref{eq:cost-c} and \eqref{eq:cost-a}, the per-iteration complexity of Algorithm \ref{alg:apg} is \begin{equation}\label{eq:cost}
 N\cdot\mathcal{O}\left(\sum_{j=1}^N\big(\prod_{i=1}^j R_i\big)\big(\prod_{i=j}^N I_i\big)+\sum_{j=1}^N\big(\prod_{i=1}^j I_i\big)\big(\prod_{i=j}^N R_i\big)\right).
\end{equation}

\section{Proof of Theorem \ref{thm:convg}}

\subsection{Subsequence convergence} First, we give a subsequence convergence result, namely, any limit point of $\{\bm{\mcw}^k\}$ is a stationary point.
Using Lemma 2.1 of \cite{xublock}, we have
\begin{align}
&F(\bm{\mcc}^{k,n-1},\mbfa_{j<n}^k,\mbfa_{j\ge n}^{k-1})-F(\bm{\mcc}^{k,n},\mbfa_{j<n}^k,\mbfa_{j\ge n}^{k-1})\cr
\ge &\frac{L_c^{k,n}}{2}\|\hat{\bm{\mcc}}^{k,n}-\bm{\mcc}^{k,n}\|_F^2+L_c^{k,n}\left\langle\hat{\bm{\mcc}}^{k,n}-\bm{\mcc}^{k,n-1},\bm{\mcc}^{k,n}-\hat{\bm{\mcc}}^{k,n}\right\rangle\cr
=&\frac{L_c^{k,n}}{2}\|\bm{\mcc}^{k,n-1}-\bm{\mcc}^{k,n}\|_F^2-\frac{L_c^{k,n}}{2}(\omega_c^{k,n})^2\|\bm{\mcc}^{k,n-2}-\bm{\mcc}^{k,n-1}\|_F^2\label{eq:keyineq}\\
\ge&\frac{L_c^{k,n}}{2}\|\bm{\mcc}^{k,n-1}-\bm{\mcc}^{k,n}\|_F^2-\frac{L_c^{k,n-1}}{2}\delta_\omega^2\|\bm{\mcc}^{k,n-2}-\bm{\mcc}^{k,n-1}\|_F^2,\label{eq:ineq-c}
\end{align}
where we have used $\omega_c^{k,n}\le\delta_\omega\sqrt{\frac{L_c^{k,n-1}}{L_c^{k,n}}}$ to get the last inequality. Note that if the re-update in Line \textbf{\ref{reupdate}} is performed, then $\omega_c^{k,n}=0$ in \eqref{eq:keyineq}, and \eqref{eq:ineq-c} still holds. Similarly, we have
\begin{equation}
\begin{array}{ll}&F(\bm{\mcc}^{k,n},\mbfa_{j<n}^k,\mbfa_{j\ge n}^{k-1})-F(\bm{\mcc}^{k,n},\mbfa_{j\le n}^k,\mbfa_{j> n}^{k-1})\\[0.1cm]
\ge&\frac{L_n^k}{2}\|\mbfa_n^{k-1}-\mbfa_n^k\|_F^2-\frac{L_n^{k-1}}{2}\delta_\omega^2\|\mbfa_n^{k-2}-\mbfa_n^{k-1}\|_F^2.\label{eq:ineq-a}
\end{array}
\end{equation}
Summing \eqref{eq:ineq-c} and \eqref{eq:ineq-a} together over $n$ and noting $\bm{\mcc}^{k,-1}=\bm{\mcc}^{k-1,N-1}, \bm{\mcc}^{k,0}=\bm{\mcc}^{k-1,N}$ yield
{\small\begin{align}
&F(\bm{\mcw}^{k-1})-F(\bm{\mcw}^{k})\cr
\ge&\sum_{n=1}^N\left(\frac{L_c^{k,n}}{2}\|\bm{\mcc}^{k,n-1}-\bm{\mcc}^{k,n}\|_F^2-\frac{L_c^{k,n-1}}{2}\delta_\omega^2\|\bm{\mcc}^{k,n-2}-\bm{\mcc}^{k,n-1}\|_F^2\right.\cr
&\hspace{0.8cm}\left.+\frac{L_n^k}{2}\|\mbfa_n^{k-1}-\mbfa_n^k\|_F^2-\frac{L_n^{k-1}}{2}\delta_\omega^2\|\mbfa_n^{k-2}-\mbfa_n^{k-1}\|_F^2\right)\cr
=&\frac{L_c^{k,N}}{2}\|\bm{\mcc}^{k,N-1}-\bm{\mcc}^{k,N}\|_F^2-\frac{L_c^{k-1,N}}{2}\delta_\omega^2\|\bm{\mcc}^{k-1,N-1}-\bm{\mcc}^{k-1,N}\|_F^2\label{eq:ineq-ca}\\
&+\sum_{n=1}^{N-1}\frac{(1-\delta_\omega^2)L_c^{k,n}}{2}\|\bm{\mcc}^{k,n-1}-\bm{\mcc}^{k,n}\|_F^2\cr
&+\sum_{n=1}^N\left(\frac{L_n^k}{2}\|\mbfa_n^{k-1}-\mbfa_n^k\|_F^2-\frac{L_n^{k-1}}{2}\delta_\omega^2\|\mbfa_n^{k-2}-\mbfa_n^{k-1}\|_F^2\right).\nonumber
\end{align}}
Summing \eqref{eq:ineq-ca} over $k$, we have
\begin{align}
&F(\bm{\mcw}^{0})-F(\bm{\mcw}^{K})\cr
\ge&\sum_{k=1}^K\sum_{n=1}^N\left(\frac{(1-\delta_\omega^2)L_c^{k,n}}{2}\|\bm{\mcc}^{k,n-1}-\bm{\mcc}^{k,n}\|_F^2+\frac{(1-\delta_\omega^2)L_n^{k}}{2}\|\mbfa_n^{k-1}-\mbfa_n^k\|_F^2\right)\cr
\ge&\frac{(1-\delta_\omega^2)L_d}{2}\sum_{k=1}^K\sum_{n=1}^N\left(\|\bm{\mcc}^{k,n-1}-\bm{\mcc}^{k,n}\|_F^2+\|\mbfa_n^{k-1}-\mbfa_n^k\|_F^2\right).
\end{align}
Letting $K\to\infty$ and observing $F$ is lower bounded, we have
\begin{equation}\label{eq:summable}\sum_{k=1}^\infty\sum_{n=1}^N\left(\|\bm{\mcc}^{k,n-1}-\bm{\mcc}^{k,n}\|_F^2+\|\mbfa_n^{k-1}-\mbfa_n^k\|_F^2\right)<\infty.
\end{equation}

Suppose $\bar{\bm{\mcw}}=(\bar{\bm{\mcc}},\bar{\mbfa}_1,\cdots,\bar{\mbfa}_N)$ is a limit point of $\{\bm{\mcw}^k\}$. Then there is a subsequence $\{\bm{\mcw}^{k'}\}$ converging to $\bar{\bm{\mcw}}$. Since $\{L_c^{k,n},L_n^k\}$ is bounded, passing another subsequence if necessary, we assume $L_c^{k',n}\to\bar{L}_c^n$ and $L_n^{k'}\to\bar{L}_n$. Note that \eqref{eq:summable} implies $\mbfa^{k'-1}\to \bar{\mbfa}$ and $\bm{\mcc}^{m,n}\to\bar{\bm{\mcc}}$ for all $n$ and $m=k',k'-1,k'-2$, as $k\to\infty$. Hence, $\hat{\bm{\mcc}}^{k',n}\to\bar{\bm{\mcc}}$ for all $n$, as $k\to\infty$. Recall that
\begin{equation}\label{min-c}
\hspace{-0.1cm}\bm{\mcc}^{k',n}=\argmin_{\bm{\mcc}\ge0}\left\langle\nabla_{\bm{\mcc}}
\ell(\hat{\bm{\mcc}}^{k',n},\mbfa_{j<n}^{k'},\mbfa_{j\ge n}^{k'-1}),\bm{\mcc}-\hat{\bm{\mcc}}^{k',n}\right\rangle+\frac{L_c^{k',n}}{2}\|\bm{\mcc}-\hat{\bm{\mcc}}^{k',n}\|_F^2+\lambda_c\|\bm{\mcc}\|_1.\hspace{-0.2cm}
\end{equation}
Letting $k\to\infty$ and using the continuity of the objective in \eqref{min-c} give
$$\bar{\bm{\mcc}}=\argmin_{\bm{\mcc}\ge0}\left\langle\nabla_{\bm{\mcc}}
\ell(\bar{\bm{\mcc}},\bar{\mbfa}),\bm{\mcc}-\bar{\bm{\mcc}}\right\rangle+\frac{\bar{L}_c^{n}}{2}\|\bm{\mcc}-\bar{\bm{\mcc}}\|_F^2+\lambda_c\|\bm{\mcc}\|_1.$$
Hence, $\bar{\bm{\mcc}}$ satisfies the first-order optimality condition
\begin{equation}\label{opt-c}
\left\langle \nabla_{\bm{\mcc}}
\ell(\bar{\bm{\mcc}},\bar{\mbfa})+\lambda_c\bm{\mathcal{P}}_c,\bm{\mcc}-\bar{\bm{\mcc}}\right\rangle\ge0,~\text{for all }\bm{\mcc}\ge0, \text{ some }\bm{\mathcal{P}}_c\in\partial\|\bar{\bm{\mcc}}\|_1.
\end{equation} 
Similarly, we have for all $n$ that
\begin{equation}\label{opt-a}
\left\langle \nabla_{\mbfa_n}
\ell(\bar{\bm{\mcc}},\bar{\mbfa})+\lambda_n\bm{\mathbf{P}}_n,\mbfa_n-\bar{\mbfa}_n\right\rangle\ge0,~\text{for all }\mbfa_n\ge0, \text{ some }\bm{\mathbf{P}}_n\in\partial\|\bar{\mbfa}_n\|_1.
\end{equation} 
Note \eqref{opt-c} together with \eqref{opt-a} gives the first-order optimality conditions of \eqref{eq:spntd}. Hence, $\bar{\bm{\mcw}}$ is a stationary point. 

\subsection{Global convergence} Next we show the entire sequence $\{\bm{\mcw}^k\}$ converges to a limit point $\bar{\bm{\mcw}}$. Since all $\lambda_c,\lambda_1,\cdots,\lambda_N$ are positive, the sequence $\{\bm{\mcw}^k\}$ is bounded and admits a finite limit point $\bar{\bm{\mcw}}$. Let $E=\{\bm{\mcw}: \|\bm{\mcw}\|_F\le 4\nu\}$, where $\|\bm{\mcw}\|_F\triangleq\sqrt{\|\bm{\mcc}\|_F^2+\|\mbfa\|_F^2}$ and $\nu$ is a constant such that $\|(\bm{\mcc}^{k,n},\mbfa^k)\|_F\le\nu$ for all $k,n$. Let $L_G$ be a uniform Lipschitz constant of $\nabla_{\bm{\mcc}}\ell(\bm{\mcw})$ and $\nabla_{\mbfa_n}\ell(\bm{\mcw}), n = 1,\cdots,N,$ over $E$, namely,
\begin{subequations}\label{lips}
\begin{align}
\|\nabla_{\bm{\mcc}}\ell(\bm{\mcy})-\nabla_{\bm{\mcc}}\ell(\bm{\mcz})\|_F\le& L_G\|\bm{\mcy}-\bm{\mcz}\|_F,~\forall \bm{\mcy},\bm{\mcz}\in E,\label{lips-c}\\
\|\nabla_{\mbfa_n}\ell(\bm{\mcy})-\nabla_{\mbfa_n}\ell(\bm{\mcz})\|_F\le& L_G\|\bm{\mcy}-\bm{\mcz}\|_F,~\forall \bm{\mcy},\bm{\mcz}\in E,~\forall n,\label{lips-a}
\end{align}
\end{subequations}
Let 
$$H(\bm{\mcc},\mbfa)=\ell(\bm{\mcc},\mbfa)+\lambda_c\|\bm{\mcc}\|_1+\delta_+(\bm{\mcc})+\overset{N}{\underset{n=1}{\sum}}\big(\lambda_n\|\mbfa_n\|_1+\delta_+(\mbfa_n)\big)$$
and 
$$r_c(\bm{\mcc})=\lambda_c\|\bm{\mcc}\|_1+\delta_+(\bm{\mcc}), \quad r_n(\mbfa_n)=\lambda_n\|\mbfa_n\|_1+\delta_+(\mbfa_n),\ n=1,\cdots,N,$$
where $\delta_+(\cdot)$ is the indicator function on nonnegative orthant, namely, it equals \emph{zero} if the argument is component-wise nonnegative and $+\infty$ otherwise.

Note that \eqref{eq:spntd} is equivalent to 
\begin{equation}\label{eq:uncon}
\min_{\bm{\mcc},\mbfa}H(\bm{\mcc},\mbfa).
\end{equation}
Recall that $H$ satisfies the KL property (see \cite{lojasiewicz1993geometrie, bolte2007lojasiewicz} for example) at $\bar{\bm{\mcw}}$, namely, there exist $\gamma,\rho>0$, $\theta\in[0,1)$, and a neighborhood $B(\bar{\bm{\mcw}},\rho)\triangleq\{\bm{\mcw}:\|\bm{\mcw}-\bar{\bm{\mcw}}\|_F\le\rho\}$ such that
\begin{equation}\label{eq:KL}
|H(\bm{\mcw})-H(\bar{\bm{\mcw}})|^\theta\le \gamma\cdot\text{dist}(\mathbf{0},\partial H(\bm{\mcw})),~\text{for all }\bm{\mcw}\in B(\bar{\bm{\mcw}},\rho).
\end{equation}

Denote $H_k=H(\bm{\mcw}^k)-H(\bar{\bm{\mcw}})$. Then $H_k\downarrow 0$. Since $\bar{\bm{\mcw}}$ is a limit point of $\{\bm{\mcw}^k\}$ and $\|\mbfa^k-\mbfa^{k+1}\|_F\to 0,\|\bm{\mcc}^{k,n-1}-\bm{\mcc}^{k,n}\|_F\to 0$ for all $k,n$ from \eqref{eq:summable}, for any $T>0$, there must exist $k_0$ such that $\bm{\mcw}^j\in B(\bar{\bm{\mcw}},\rho), j=k_0,k_0+1,k_0+2$ and
\begin{align*}
&T\big(H_{k_0}^{1-\theta}+\|\mbfa^{k_0}-\mbfa^{k_0+1}\|_F+\|\mbfa^{k_0+1}-\mbfa^{k_0+2}\|_F+\|\bm{\mcc}^{k_0+2,N-1}-\bm{\mcc}^{k_0+2,N}\|_F\big)\\
&\hspace{0.3cm}+\|\bm{\mcw}^{k_0+2}-\bar{\bm{\mcw}}\|_F<\rho.
\end{align*}
Take $T$ as specified in \eqref{eq:T} and consider the sequence $\{\bm{\mcw}^k\}_{k\ge k_0}$, which is equivalent to starting the algorithm from $\bm{\mcw}^{k_0}$ and, thus without loss of generality, let $k_0=0$, namely, $\bm{\mcw}^j\in B(\bar{\bm{\mcw}},\rho), j=0,1,2$, and
\begin{equation}\label{cond1}T\big(H_{0}^{1-\theta}+\|\mbfa^{0}-\mbfa^{1}\|_F+\|\mbfa^{1}-\mbfa^{2}\|_F+\|\bm{\mcc}^{2,N-1}-\bm{\mcc}^{2,N}\|_F\big)+\|\bm{\mcw}^{2}-\bar{\bm{\mcw}}\|_F<\rho.
\end{equation}

The idea of our proof is to show 
\begin{equation}\label{belong}
\bm{\mcw}^k\in B(\bar{\bm{\mcw}},\rho),~\text{for all }k,
\end{equation}
and employ the KL inequality \eqref{eq:KL} to show $\{\bm{\mcw}^k\}$ is a Cauchy sequence, thus the entire sequence converges. Assume $\bm{\mcw}^k\in B(\bar{\bm{\mcw}},\rho)$ for $0\le k\le K$. We go to show $\bm{\mcw}^{K+1}\in B(\bar{\bm{\mcw}},\rho)$ and conclude \eqref{belong} by induction.

Note that
\begin{align*}
&\partial H(\bm{\mcw}^k)=\left\{\partial r_1(\mbfa_1^k)
+\nabla_{\mbfa_1}\ell(\bm{\mcw}^k)\right\}\times\cdots\times\left\{\partial r_N(\mbfa_N^k)
+\nabla_{\mbfa_N}\ell(\bm{\mcw}^k)\right\}\times\left\{\partial r_c(\bm{\mcc}^{k,N})
+\nabla_{\bm{\mcc}}\ell(\bm{\mcw}^k)\right\},
\end{align*}
and for all $n$ and $k$
\begin{align*}
-L_n^k(\mbfa_n^k-\hat{\mbfa}_n^k)-\nabla_{\mbfa_n}\ell(\bm{\mcc}^{k,n},\mbfa_{j<n}^k,\hat{\mbfa}_n^k,\mbfa_{j\ge n}^{k-1})+\nabla_{\mbfa_n}\ell(\bm{\mcw}^k)\in&\partial r_n(\mbfa_n^k)
+\nabla_{\mbfa_n}\ell(\bm{\mcw}^k),\\
-L_c^{k,N}(\bm{\mcc}^{k,N}-\hat{\bm{\mcc}}^{k,N})-\nabla_{\bm{\mcc}}\ell(\hat{\bm{\mcc}}^{k,N},\mbfa_{j<N}^k,{\mbfa}_N^{k-1})+\nabla_{\bm{\mcc}}\ell(\bm{\mcw}^k)\in&\partial r_c(\bm{\mcc}^{k,N})
+\nabla_{\bm{\mcc}}\ell(\bm{\mcw}^k).
\end{align*}
Hence, for all $k\le K$,
{\small\begin{align}
&\text{dist}\big(\mathbf{0},\partial H(\bm{\mcw}^k)\big)\cr
\le&\big\|(L_1^k(\mbfa_1^k-\hat{\mbfa}_1^k),\cdots,L_1^k(\mbfa_1^k-\hat{\mbfa}_1^k),L_c^{k,n}(\bm{\mcc}^{k,N}-\hat{\bm{\mcc}}^{k,N}))\big\|_F\cr
&+\sum_{n=1}^N\big\|\nabla_{\mbfa_n}\ell(\bm{\mcc}^{k,n},\mbfa_{j<n}^k,\hat{\mbfa}_n^k,\mbfa_{j\ge n}^{k-1})-\nabla_{\mbfa_n}\ell(\bm{\mcw}^k)\big\|_F\cr
&+\big\|\nabla_{\bm{\mcc}}\ell(\hat{\bm{\mcc}}^{k,N},\mbfa_{j<N}^k,{\mbfa}_N^{k-1})-\nabla_{\bm{\mcc}}\ell(\bm{\mcw}^k)\big\|_F\cr
\le&L_u\big(\|\mbfa^k-\mbfa^{k-1}\|_F+\|\mbfa^{k-1}-\mbfa^{k-2}\|_F\big)+L_u\big(\|\bm{\mcc}^{k,N}-\bm{\mcc}^{k,N-1}\|_F+\|\bm{\mcc}^{k,N-1}-\bm{\mcc}^{k,N-2}\|_F\big)\cr
&+\sum_{n=1}^NL_G\big(\|\bm{\mcc}^{k,n}-\bm{\mcc}^{k,N}\|_F+\|\mbfa^k-\mbfa^{k-1}\|_F+\|\mbfa^{k-1}-\mbfa^{k-2}\|_F\big)\cr
&+L_G\big(\|\bm{\mcc}^{k,N}-\bm{\mcc}^{k,N-1}\|_F+\|\bm{\mcc}^{k,N-1}-\bm{\mcc}^{k,N-2}\|_F+\|\mbfa^k-\mbfa^{k-1}\|_F\big)\cr
\le&\big(L_u+(N+1)L_G\big)\left(\|\mbfa^k-\mbfa^{k-1}\|_F+\|\mbfa^{k-1}-\mbfa^{k-2}\|_F\right.\label{eq:ineq3}\\
&\hspace{0.8cm}\left.+\|\bm{\mcc}^{k,N}-\bm{\mcc}^{k,N-1}\|_F+\sum_{n=1}^{N-1}\|\bm{\mcc}^{k,n-1}-\bm{\mcc}^{k,n}\|_F\right),\nonumber
\end{align}}where we have used $L_n^k,L_c^{k,n}\le L_u,~\forall k, n$ and \eqref{lips} to have the second inequality, and the third inequality is obtained from $\|\bm{\mcc}^{k,n}-\bm{\mcc}^{k,N}\|_F\le\sum_{i=n}^{N-1}\|\bm{\mcc}^{k,i}-\bm{\mcc}^{k,i+1}\|_F$ and doing some simplification.
Using the KL inequality \eqref{eq:KL} at $\bm{\mcw}=\bm{\mcw}^k$ and the inequality
$$\frac{s^\theta}{1-\theta}(s^{1-\theta}-t^{1-\theta})\ge s-t,~\forall s,t\ge0,$$
we get
\begin{equation}\label{eq:ineq4}
\frac{\gamma}{1-\theta}\text{dist}(\mathbf{0},\partial H(\bm{\mcw}^k))(H_k^{1-\theta}-H_{k+1}^{1-\theta})\ge H_k-H_{k+1}.
\end{equation}
By \eqref{eq:ineq-ca}, we have
\begin{align}
H_k-H_{k+1}
\ge&\frac{L_c^{k+1,N}}{2}\|\bm{\mcc}^{k+1,N-1}-\bm{\mcc}^{k+1,N}\|_F^2-\frac{L_c^{k,N}}{2}\delta_\omega^2\|\bm{\mcc}^{k,N-1}-\bm{\mcc}^{k,N}\|_F^2\label{eq:ineq-ca2}\\
&+\sum_{n=1}^{N-1}\frac{(1-\delta_\omega^2)L_c^{k+1,n}}{2}\|\bm{\mcc}^{k+1,n-1}-\bm{\mcc}^{k+1,n}\|_F^2\cr
&+\sum_{n=1}^N\left(\frac{L_n^{k+1}}{2}\|\mbfa_n^{k}-\mbfa_n^{k+1}\|_F^2-\frac{L_n^{k}}{2}\delta_\omega^2\|\mbfa_n^{k-1}-\mbfa_n^{k}\|_F^2\right).\nonumber
\end{align}
Combining \eqref{eq:ineq3}, \eqref{eq:ineq4}, \eqref{eq:ineq-ca2} and noting $L_c^{k+1,n}\ge L_d$ yield
{\begin{align}
&\frac{\gamma(L_u+(N+1)L_G)}{1-\theta}(H_k^{1-\theta}-H_{k+1}^{1-\theta})\big[\|\mbfa^k-\mbfa^{k-1}\|_F+\|\mbfa^{k-1}-\mbfa^{k-2}\|_F\cr
&\hspace{4cm}+\|\bm{\mcc}^{k,N}-\bm{\mcc}^{k,N-1}\|_F+\sum_{n=1}^{N-1}\|\bm{\mcc}^{k,n-1}-\bm{\mcc}^{k,n}\|_F\big]\cr
&+\delta_\omega^2\left\|(\sqrt{L_1^k}\mbfa_1^{k-1},\cdots,\sqrt{L_N^k}\mbfa_N^{k-1},\sqrt{L_c^{k,N}}\bm{\mcc}^{k,N-1})\right.\cr&\hspace{1cm}\left.-(\sqrt{L_1^k}\mbfa_1^{k},\cdots,\sqrt{L_N^k}\mbfa_N^{k},\sqrt{L_c^{k,N}}\bm{\mcc}^{k,N})\right\|_F^2\cr
\ge&\left\|(\sqrt{L_1^{k+1}}\mbfa_1^{k},\cdots,\sqrt{L_N^{k+1}}\mbfa_N^{k},\sqrt{L_c^{k+1,N}}\bm{\mcc}^{k+1,N-1})\right.\label{eq:ineq5}\\
&\hspace{0.4cm}\left.-(\sqrt{L_1^{k+1}}\mbfa_1^{k+1},\cdots,\sqrt{L_N^{k+1}}\mbfa_N^{k+1},\sqrt{L_c^{k+1,N}}\bm{\mcc}^{k+1,N})\right\|_F^2\nonumber\\
&+\frac{(1-\delta_\omega^2)L_d}{2}\sum_{n=1}^{N-1}\|\bm{\mcc}^{k+1,n-1}-\bm{\mcc}^{k+1,n}\|_F^2.\nonumber
\end{align}
By Cauchy-Schwart inequality, we estimate 
{\small\begin{align}
&\sqrt{\text{right side of inequality \eqref{eq:ineq5}}}\nonumber\\
\ge&\frac{1+\delta_\omega}{2}\left\|(\sqrt{L_1^{k+1}}\mbfa_1^{k},\cdots,\sqrt{L_N^{k+1}}\mbfa_N^{k},\sqrt{L_c^{k+1,N}}\bm{\mcc}^{k+1,N-1})\right.\cr
&\hspace{1.2cm}\left.-(\sqrt{L_1^{k+1}}\mbfa_1^{k+1},\cdots,\sqrt{L_N^{k+1}}\mbfa_N^{k+1},\sqrt{L_c^{k+1,N}}\bm{\mcc}^{k+1,N})\right\|_F\nonumber\\
&+\eta\sum_{n=1}^{N-1}\|\bm{\mcc}^{k+1,n-1}-\bm{\mcc}^{k+1,n}\|_F,\label{eq:ineq7}
\end{align}}
where $\eta>0$ is sufficiently small and depends on $\delta_\omega,L_d,N$,
and
\begin{align}
&\sqrt{\text{left side of inequality \eqref{eq:ineq5}}}\cr
\le&\frac{\mu\gamma(L_u+(N+1)L_G)}{4(1-\theta)}(H_k^{1-\theta}-H_{k+1}^{1-\theta})\label{eq:ineq6}\\
&+\frac{1}{\mu}\big[\|\mbfa^k-\mbfa^{k-1}\|_F+\|\mbfa^{k-1}-\mbfa^{k-2}\|_F+\|\bm{\mcc}^{k,N}-\bm{\mcc}^{k,N-1}\|_F+\sum_{n=1}^{N-1}\|\bm{\mcc}^{k,n-1}-\bm{\mcc}^{k,n}\|_F\big]\cr
&+\delta_\omega\left\|(\sqrt{L_1^k}\mbfa_1^{k-1},\cdots,\sqrt{L_N^k}\mbfa_N^{k-1},\sqrt{L_c^{k,N}}\bm{\mcc}^{k,N-1})-(\sqrt{L_1^k}\mbfa_1^{k},\cdots,\sqrt{L_N^k}\mbfa_N^{k},\sqrt{L_c^{k,N}}\bm{\mcc}^{k,N})\right\|_F,\nonumber
\end{align}
where $\mu>0$ is a sufficiently large constant such that $\frac{1}{\mu}<\min(\eta,\frac{1-\delta_\omega}{4}\sqrt{\frac{L_d}{2}}).$
Combining \eqref{eq:ineq5},\eqref{eq:ineq6}, \eqref{eq:ineq7} and summing them over $k$ from $2$ to $K$ give
\begin{align}
&\frac{\mu\gamma(L_u+(N+1)L_G)}{4(1-\theta)}(H_2^{1-\theta}-H_{K+1}^{1-\theta})\cr
&+\frac{1}{\mu}\sum_{k=2}^K\big[\|\mbfa^k-\mbfa^{k-1}\|_F+\|\mbfa^{k-1}-\mbfa^{k-2}\|_F+\|\bm{\mcc}^{k,N}-\bm{\mcc}^{k,N-1}\|_F+\sum_{n=1}^{N-1}\|\bm{\mcc}^{k,n-1}-\bm{\mcc}^{k,n}\|_F\big]\cr
&+\delta_\omega\sum_{k=2}^K\left\|(\sqrt{L_1^k}\mbfa_1^{k-1},\cdots,\sqrt{L_N^k}\mbfa_N^{k-1},\sqrt{L_c^{k,N}}\bm{\mcc}^{k,N-1})-(\sqrt{L_1^k}\mbfa_1^{k},\cdots,\sqrt{L_N^k}\mbfa_N^{k},\sqrt{L_c^{k,N}}\bm{\mcc}^{k,N})\right\|_F\cr
\ge &\frac{1+\delta_\omega}{2}\sum_{k=2}^K\left\|(\sqrt{L_1^{k+1}}\mbfa_1^{k},\cdots,\sqrt{L_N^{k+1}}\mbfa_N^{k},\sqrt{L_c^{k+1,N}}\bm{\mcc}^{k+1,N-1})\right.\cr
&\hspace{1cm}\left.-(\sqrt{L_1^{k+1}}\mbfa_1^{k+1},\cdots,\sqrt{L_N^{k+1}}\mbfa_N^{k+1},\sqrt{L_c^{k+1,N}}\bm{\mcc}^{k+1,N})\right\|_F+\eta\sum_{k=2}^K\sum_{n=1}^{N-1}\|\bm{\mcc}^{k+1,n-1}-\bm{\mcc}^{k+1,n}\|_F.\nonumber
\end{align}
Simplifying the above inequality, we have
\begin{align}
&\frac{\mu\gamma(L_u+(N+1)L_G)}{4(1-\theta)}(H_2^{1-\theta}-H_{K+1}^{1-\theta})\nonumber\\
&+\frac{1}{\mu}\sum_{k=2}^K\left(\|\mbfa^k-\mbfa^{k-1}\|_F+\|\mbfa^{k-1}-\mbfa^{k-2}\|_F+\|\bm{\mcc}^{k,N}-\bm{\mcc}^{k,N-1}\|_F\right)\nonumber\\
&+\delta_\omega\big\|(\sqrt{L_1^2}\mbfa_1^{1},\cdots,\sqrt{L_N^2}\mbfa_N^{1},\sqrt{L_c^{2,N}}\bm{\mcc}^{2,N-1})-(\sqrt{L_1^2}\mbfa_1^{2},\cdots,\sqrt{L_N^2}\mbfa_N^{2},\sqrt{L_c^{2,N}}\bm{\mcc}^{2,N})\big\|_F\nonumber\\
\ge &\frac{1+\delta_\omega}{2}\left\|(\sqrt{L_1^{K+1}}\mbfa_1^{K},\cdots,\sqrt{L_N^{K+1}}\mbfa_N^{K},\sqrt{L_c^{K+1,N}}\bm{\mcc}^{K+1,N-1})\right.\label{eq:ineq8}\\
&\hspace{1.5cm}\left.-(\sqrt{L_1^{K+1}}\mbfa_1^{K+1},\cdots,\sqrt{L_N^{K+1}}\mbfa_N^{K+1},\sqrt{L_c^{K+1,N}}\bm{\mcc}^{K+1,N})\right\|_F\nonumber\\
&+\frac{1-\delta_\omega}{2}\sum_{k=2}^{K-1}\left\|(\sqrt{L_1^{k+1}}\mbfa_1^{k},\cdots,\sqrt{L_N^{k+1}}\mbfa_N^{k},\sqrt{L_c^{k+1,N}}\bm{\mcc}^{k+1,N-1})\right.\nonumber\\
&\hspace{2cm}\left.-(\sqrt{L_1^{k+1}}\mbfa_1^{k+1},\cdots,\sqrt{L_N^{k+1}}\mbfa_N^{k+1},\sqrt{L_c^{k+1,N}}\bm{\mcc}^{k+1,N})\right\|_F\nonumber\\
&+(\eta-\frac{1}{\mu})\sum_{k=2}^K\sum_{n=1}^{N-1}\|\bm{\mcc}^{k+1,n-1}-\bm{\mcc}^{k+1,n}\|_F.\nonumber
\end{align}
Note that 
\begin{align}
&\left\|(\sqrt{L_1^{k+1}}\mbfa_1^{k},\cdots,\sqrt{L_N^{k+1}}\mbfa_N^{k},\sqrt{L_c^{k+1,N}}\bm{\mcc}^{k+1,N-1})\right.\cr
&\hspace{0.3cm}\left.-(\sqrt{L_1^{k+1}}\mbfa_1^{k+1},\cdots,\sqrt{L_N^{k+1}}\mbfa_N^{k+1},\sqrt{L_c^{k+1,N}}\bm{\mcc}^{k+1,N})\right\|_F^2\nonumber\\
=&\sum_{n=1}^NL_n^{k+1}\|\mbfa_n^k-\mbfa_n^{k+1}\|_F^2+L_c^{k+1,N}\|\bm{\mcc}^{k+1,N-1}-\bm{\mcc}^{k+1,N}\|_F^2\nonumber\\
\ge&L_d(\|\mbfa^k-\mbfa^{k+1}\|_F^2+\|\bm{\mcc}^{k+1,N-1}-\bm{\mcc}^{k+1,N}\|_F^2\nonumber\\
\ge&\frac{L_d}{2}(\|\mbfa^k-\mbfa^{k+1}\|_F+\|\bm{\mcc}^{k+1,N-1}-\bm{\mcc}^{k+1,N}\|_F)^2\label{eq:ineq9}
\end{align}
Plugging \eqref{eq:ineq9} to inequality \eqref{eq:ineq8} gives
\begin{align}
&\frac{\mu\gamma(L_u+(N+1)L_G)}{4(1-\theta)}(H_2^{1-\theta}-H_{K+1}^{1-\theta})\cr
&+\frac{1}{\mu}\sum_{k=2}^K\left(\|\mbfa^k-\mbfa^{k-1}\|_F+\|\mbfa^{k-1}-\mbfa^{k-2}\|_F+\|\bm{\mcc}^{k,N}-\bm{\mcc}^{k,N-1}\|_F\right)\nonumber\\
&+\delta_\omega\|(\sqrt{L_1^2}\mbfa_1^{1},\cdots,\sqrt{L_N^2}\mbfa_N^{1},\sqrt{L_c^{2,N}}\bm{\mcc}^{2,N-1})-(\sqrt{L_1^2}\mbfa_1^{2},\cdots,\sqrt{L_N^2}\mbfa_N^{2},\sqrt{L_c^{2,N}}\bm{\mcc}^{2,N})\|_F\nonumber\\
\ge&\frac{1+\delta_\omega}{2}\sqrt{\frac{L_d}{2}}(\|\mbfa^K-\mbfa^{K+1}\|_F+\|\bm{\mcc}^{K+1,N-1}-\bm{\mcc}^{K+1,N}\|_F)\nonumber\\
&+\frac{1-\delta_\omega}{2}\sqrt{\frac{L_d}{2}}\sum_{k=2}^{K-1}(\|\mbfa^k-\mbfa^{k+1}\|_F+\|\bm{\mcc}^{k+1,N-1}-\bm{\mcc}^{k+1,N}\|_F)\nonumber\\
&+(\eta-\frac{1}{\mu})\sum_{k=2}^K\sum_{n=1}^{N-1}\|\bm{\mcc}^{k+1,n-1}-\bm{\mcc}^{k+1,n}\|_F,\nonumber
\end{align}
which implies by noting $H_0\ge H_k\ge 0$, $\bm{\mcc}^{k+1,0}=\bm{\mcc}^{k,N}$ and $L_n^k,L_c^{k,n}\le L_u,~\forall k,n$ that
\begin{align}
&\frac{\mu\gamma(L_u+(N+1)L_G)}{4(1-\theta)}H_0^{1-\theta}+\frac{1}{\mu}\left(2\|\mbfa^1-\mbfa^{2}\|_F+\|\mbfa^{0}-\mbfa^{1}\|_F+\|\bm{\mcc}^{2,N}-\bm{\mcc}^{2,N-1}\|_F\right)\nonumber\\
&+\delta_\omega\sqrt{L_u}\big(\|\mbfa^1-\mbfa^2\|_F+\|\bm{\mcc}^{2,N-1}-\bm{\mcc}^{2,N}\|_F\big)\nonumber\\
\ge&\frac{1+\delta_\omega}{2}\sqrt{\frac{L_d}{2}}\big(\|\mbfa^K-\mbfa^{K+1}\|_F+\|\bm{\mcc}^{K+1,N-1}-\bm{\mcc}^{K+1,N}\|_F\big)\nonumber\\
&+(\frac{1-\delta_\omega}{2}\sqrt{\frac{L_d}{2}}-\frac{2}{\mu})\sum_{k=2}^{K-1}\big(\|\mbfa^k-\mbfa^{k+1}\|_F+\|\bm{\mcc}^{k+1,N-1}-\bm{\mcc}^{k+1,N}\|_F\big)\nonumber\\
&+(\eta-\frac{1}{\mu})\sum_{k=2}^K\|\bm{\mcc}^{k,N}-\bm{\mcc}^{k+1,N-1}\|_F,\nonumber\\
\ge&\tau\big(\|\mbfa^K-\mbfa^{K+1}\|_F+\|\bm{\mcc}^{K,N}-\bm{\mcc}^{K+1,N}\|_F\big)\label{eq:ineq10}\\
&+\tau\sum_{k=2}^{K-1}\big(\|\mbfa^k-\mbfa^{k+1}\|_F+\|\bm{\mcc}^{k,N}-\bm{\mcc}^{k+1,N}\|_F\big),\nonumber
\end{align}
where 
$\tau=\min\left(\frac{1-\delta_\omega}{2}\sqrt{\frac{L_d}{2}}-\frac{2}{\mu},~\eta-\frac{1}{\mu}\right).$
Let 
\begin{equation}\label{eq:T}
T=\max\left(\frac{\mu\gamma(L_u+(N+1)L_G)}{4\tau(1-\theta)},~\frac{1}{2\mu\tau}+\frac{\delta_\omega}{\tau}\sqrt{L_u}\right).
\end{equation}
Then \eqref{eq:ineq10} implies
\begin{align}
&T\big(H_0^{1-\theta}+\|\mbfa^0-\mbfa^1\|_F+\|\mbfa^1-\mbfa^2\|_F+\|\bm{\mcc}^{2,N-1}-\bm{\mcc}^{2,N}\|_F\big)\nonumber\\
\ge&\|\bm{\mcw}^K-\bm{\mcw}^{K+1}\|_F+\sum_{k=2}^{K-1}\|\bm{\mcw}^k-\bm{\mcw}^{k+1}\|_F\label{eq:ineq11},
\end{align}
from which we have
\begin{align*}
&\|\bm{\mcw}^{K+1}-\bar{\bm{\mcw}}\|_F\\
\le&\|\bm{\mcw}^K-\bm{\mcw}^{K+1}\|_F+\sum_{k=2}^{K-1}\|\bm{\mcw}^k-\bm{\mcw}^{k+1}\|_F+\|\bm{\mcw}^{2}-\bar{\bm{\mcw}}\|_F\\
\le &T\big(H_0^{1-\theta}+\|\mbfa^0-\mbfa^1\|_F+\|\mbfa^1-\mbfa^2\|_F+\|\bm{\mcc}^{2,N-1}-\bm{\mcc}^{2,N}\|_F\big)+\|\bm{\mcw}^{2}-\bar{\bm{\mcw}}\|_F<\rho.
\end{align*}
Hence, $\bm{\mcw}^{K+1}\in B(\bar{\bm{\mcw}},\rho)$. By induction, we have $\bm{\mcw}^{k}\in B(\bar{\bm{\mcw}},\rho)$ for all $k$, so \eqref{eq:ineq11} holds for all $K$. Letting $K\to\infty$ gives $\sum_{k=2}^{\infty}\|\bm{\mcw}^k-\bm{\mcw}^{k+1}\|_F<\infty$, namely, $\{\bm{\mcw}^k\}$ is a Cauchy sequence and, thus $\bm{\mcw}^k$ converges. Since $\bar{\bm{\mcw}}$ is a limit point of $\{\bm{\mcw}^k\}$, then $\bm{\mcw}^k\to \bar{\bm{\mcw}}$. This completes the proof.

\end{document}